\documentclass[11pt]{article}

\usepackage[margin=1in]{geometry}
\usepackage{amssymb, color}
\usepackage{hyperref}
\definecolor{labelkey}{rgb}{0,0,1}
\usepackage{tikz}
 \usepackage{pgfplots}
\usepackage{graphicx,xcolor}
\usepackage{amsthm,amsfonts,euscript,amsmath,comment,slashed,here,todonotes}
\usepackage[affil-it]{authblk}

\newtheorem{theorem}{Theorem}[section]
\newtheorem{lemma}[theorem]{Lemma}
\newtheorem{proposition}{Proposition}[section]
\newtheorem{corollary}{Corollary}[section]

\theoremstyle{definition}
\newtheorem{definition}[theorem]{Definition}

\theoremstyle{remark}

\numberwithin{equation}{section}

\newcommand{\ud}{\,\mathrm{d}}
\newcommand{\p}{\ensuremath{\partial}}
\newcommand{\n}{\ensuremath{\nonumber}}
\newcommand{\eps}{\ensuremath{\varepsilon}}

\newcommand{\bigO}{\mathcal{O}}

\newcommand\be{\begin{equation}}
\newcommand\ee{\end{equation}}
\newcommand\bea{\begin{eqnarray}}
\newcommand\eea{\end{eqnarray}}
\newcommand\bi{\begin{itemize}}
\newcommand\ei{\end{itemize}}
\newcommand\ben{\begin{enumerate}}
\newcommand\bena{\begin{enumerate}[(a)]}
\newcommand\een{\end{enumerate}}

\newcommand\bp{\begin{proof}}
\newcommand\ep{\end{proof}}

\allowdisplaybreaks

\title{Regularity and Expansion for Steady Prandtl Equations}
\author{Yan Guo \footnote{Division of Applied Mathematics, Brown University, \url{yan_guo@brown.edu}} \qquad Sameer Iyer \footnote{Department of Mathematics, Princeton University, \url{ssiyer@math.princeton.edu}}}

\begin{document}

\maketitle

\begin{abstract}
Due to degeneracy near the boundary, the question of high regularity for solutions to the steady Prandtl equations has been a longstanding open question since the celebrated work of Oleinik. We settle this open question in affirmative in the absence of an external pressure. Our method is based on energy estimates for the quotient, $q = \frac{v}{\bar{u}}$, $\bar{u}$ being the classical Prandtl solution, via the linear Derivative Prandtl Equation (LDP). As a consequence, our regularity result leads to the construction of Prandtl layer expansion up to any order.
\end{abstract}

\section{Introduction}

\subsection{Steady Prandtl Equation}

The 2D, steady, Prandtl equations are given by 
\begin{align} \label{Pr.leading}
\left.
\begin{aligned} 
&\bar{u} \bar{u}_x + \bar{v} \bar{u}_y - \bar{u}_{yy} = - \p_x p_E(x, 0) \\
&\bar{u}_x + \bar{v}_y = 0
\end{aligned}
\right\} \text{ on } (x, y) \in (0, L) \times \mathbb{R}_+.
\end{align}
The quantity $p_E(x, 0)$ above is considered prescribed, and $\p_x p_E(x, 0)$ evidently acts as a forcing term to the Prandtl equations. In this paper, we are concerned with the homogeneous Prandtl equations, that is $p_E = 0$.  

The Prandtl equations \eqref{Pr.leading} are thought of as an \textit{evolution} equation, with $x$ being a time-like variable, and $y$ being space-like. This can be formally seen through a crude derivative count $\bar{u} \p_x \approx \p_{yy}$, which indicates \eqref{Pr.leading} is a quasilinear, parabolic equation. As a result, the equations \eqref{Pr.leading} are supplemented with boundary conditions at $y = 0, y \rightarrow \infty$, and an initial condition at $x = 0$. The quantity $L$ appearing in \eqref{Pr.leading} is thought of as the time over which we are considering the evolution. To summarize the datum that is prescribed is 
\begin{align} \label{Pr.leading:1}
\bar{u}|_{y = 0} = \bar{v}|_{y = 0} = 0, \qquad \bar{u}|_{y \uparrow \infty} = u_E(\infty), \qquad \bar{u}|_{x = 0} = \bar{U}_0(y). 
\end{align} 
The conditions at $y = 0$ are called the no-slip condition, which is the classical boundary condition for viscous flows. The condition at $y = \infty$ is classically known as the Euler matching condition, and physically describes the horizontal component, $u$, matching up with some ambient Euler flow at $\infty$. Thus, the quantity $u_E(\infty)$ is a prescribed constant. We can, without loss of generality, assume this constant is $1$. 

The Prandtl system \eqref{Pr.leading} - \eqref{Pr.leading:1} is a celebrated system of PDEs that arise naturally when considering the physically important problem of describing the vanishing viscosity limit of Navier-Stokes flows on domains with boundaries. The precise link to the inviscid limit problem will be made clear later in the introduction, in subsection \ref{subsection:Inviscid:Limit}. In addition, the Prandtl system by itself poses significant mathematical challenges that have attracted substantial attention over the years, beginning with the seminal work of Oleinik, \cite{Oleinik}.

The following local regularity result for the Prandtl system, \eqref{Pr.leading} - \eqref{Pr.leading:1},  is classical (\cite{Oleinik}, P. 21, Theorem 2.1.1): 
\begin{theorem}[Oleinik]   \label{thm.Oleinik} Assume $\bar{U}_0 \in C^\infty(\mathbb{R}_+)$ satisfies the following
\begin{align} 
\begin{aligned} \label{OL.1}
&\bar{U}_0(y) > 0 \text{ for } y > 0, \qquad \bar{U}_0'(0) > 0, \qquad \bar{U}_0''(0) = \bar{U}_0'''(0) = 0, 
\end{aligned}
\end{align}

\noindent  Then for some $L > 0$, there exists a solution, $[\bar{u}, \bar{v}]$ to \eqref{Pr.leading} - \eqref{Pr.leading:1} satisfying, for some $y_0, m_0 > 0$, 
\begin{align} \label{coe.2}
&\sup_{x \in (0,L)} \sup_{y \in (0, y_0)} |\bar{u}, \bar{v}, \bar{u}_{y}, \bar{u}_{yy}, \bar{u}_{x}| \lesssim C(\bar{U}_0), \\ \label{coe.1}
&\sup_{x \in (0,L)} \sup_{y \in (0, y_0)} \bar{u}_{y} > m_0 > 0. 
\end{align}
\end{theorem}

The method employed by Oleinik to prove Theorem \ref{thm.Oleinik} is to pass to the following change of coordinates, known as the von-Mise transform: 
\begin{align} \label{von.Mise}
&(x,\psi) = (x,\int_0^y \bar{u}(x,y') \ud y'), \hspace{3 mm} \tilde{u}(x,\psi) := \bar{u}(x,y), \hspace{3 mm} \p_x \tilde{u} = \p_\psi \{ \tilde{u} \tilde{u}_\psi \}. 
\end{align}

\noindent The equation above is quasilinear, degenerate diffusion equation. Estimates (\ref{coe.2}) - (\ref{coe.1}) are subsequently proven using maximum principle techniques. 

Despite this, establishing \textit{higher} regularity has been an important open problem. One cannot simply differentiate the equation (\ref{von.Mise}) and repeat Oleinik's argument because the commutators that arise from this process are uncontrollable near the boundary $\{\psi = 0\} = \{y = 0\}$. One contribution of the present paper is to resolve this problem by introducing a novel set of energy estimates enabling us to proving higher regularity estimates.

\begin{theorem} \label{thm.main} Assume the data $\bar{U}_0(y)$ is provided satisfying conditions (\ref{OL.1}). Assume also generic compatibility conditions at the corner $(0,0)$ (as defined explicitly in Definition \ref{def:compatibility:condition}) up to order $M_0$. Fix $N > 0$ sufficiently large relative to universal constants, and assume the exponential decay of derivatives: 
\begin{align}
\| \p_y^{j} \bar{U}_0 e^{Ny} \|_{L^\infty} \le C_0 \text{ for } j = 1,...,M_0.
\end{align}
Then there exists an $0 < L << 1$ depending on $C_0$ so that on $0 \le x \le L$, Oleinik's solutions  guaranteed by Theorem \ref{thm.Oleinik} obey the following estimates for $2\alpha +  \beta \le M_0$ and for $\gamma \le \frac{M_0}{2} - 1$: 
\begin{align*}
\| \p_x^\alpha \p_y^\beta \bar{u} e^{Ny} \|_{L^\infty} + \|\p_x^{\gamma} \bar{v} \|_{L^\infty}   \lesssim C_0. 
\end{align*}
\end{theorem}

The system governing the behavior of the higher order derivatives of $[\bar{u}, \bar{v}]$ is, naturally, the \textit{linearized} Prandtl equations. We now state precisely the linearized Prandtl system we consider, where we will work with some more generality than is strictly required for the proof of Theorem \ref{thm.main}. Indeed, we consider the following problem for the unknown vector field, $[u_p, v_p]$: 
\begin{align}
\begin{aligned} \label{Orig.Orig.Pr.intro}
&\bar{u} \p_x u_p + u_p \p_x \bar{u} + \bar{v} \p_y u_p + v_p \p_y \bar{u} - \p_{yy} u_p = f, \qquad (x, y) \in (0, L) \times \mathbb{R}_+ \\
&\p_x u_p + \p_y v_p = 0, \qquad (x, y) \in (0, L) \times \mathbb{R}_+ \\
&[u_p, v_p]|_{y = 0} =  [u_0(x), 0], \qquad \lim_{y \rightarrow \infty} u_p(x, y) = 0, \qquad v_p|_{x = 0} = V_0(y). 
\end{aligned}
\end{align}

\noindent The boundary condition $u_0(x)$ is a prescribed, smooth function (which certainly is allowed to be zero). We note that in our prescription of \eqref{Orig.Orig.Pr.intro}, the initial datum is now prescribed at the level of $v_p|_{x = 0}$ as opposed to $u_p$ as in (\ref{Pr.leading}). 

The method that we use to obtain Theorem \ref{thm.main} is sufficiently robust also to obtain estimates for (\ref{Orig.Orig.Pr.intro}). This is an important distinction between our technique and that of Oleinik. Indeed, the proof of Theorem \ref{thm.Oleinik} relies on a very specific use of the maximum principle which cannot apply to the more general, linearized problem. We now state our main result regarding the linearized Prandtl system, \eqref{Orig.Orig.Pr.intro}.

\begin{theorem}  \label{thm:LP} Fix any $M_0$ and an $N$ large relative to universal constants. Assume compatibility conditions on the initial datum, $V_0$, as specified in Definitions \ref{def:compatibility:condition} and \ref{def:integral:cond}. Assume also that 
\begin{align} \label{assume:V0:thm:LP}
\|\p_y^{k} V_0 e^{Ny} \|_{L^2(\mathbb{R}_+)} \le C_0 < \infty \text{ for } k = 1,...,2M_0 + 3. 
\end{align}
Assume the forcing, $f$, satisfies the bounds
\begin{align}
\| \p_{x}^{k+1} \p_y f e^{Ny} \|_{L^2} + \| \p_x^{M_0 + 2} \p_y f \langle y \rangle \|_{L^2} \le C_1 < \infty \text{ for } k = 0,...,M_0. 
\end{align}
Then the following estimates are valid
\begin{align}
 &\| \p_x^\alpha v_p \|_{L^\infty} \lesssim C(C_0, C_1) \text{ for } 0 \le \alpha \le M_0 + 1, \\
 &\| \p_x^{\alpha} \p_y^{\beta} v_p e^{Ny} \|_{L^\infty} \lesssim C(C_0, C_1) \text{ for } \alpha + \frac{\beta}{2} \le M_0 + 1, \\
 &\| \p_x^{\alpha} \p_y^\beta u_p e^{Ny} \|_{L^\infty} \lesssim C(C_0, C_1) \text{ for } \alpha + \frac{\beta + 1}{2} \le M_0 + 1.  
\end{align}
\end{theorem}

\subsection{Inviscid limit of Navier-Stokes} \label{subsection:Inviscid:Limit}

We now describe the link of Theorems \ref{thm.main} and \ref{thm:LP} with the related and very important problem of describing the vanishing viscosity limit of Navier-Stokes flows. Consider the steady, incompressible Navier-Stokes equations with viscosity $\eps$ posed on the domain
\begin{align*}
\Omega := (0, L) \times \mathbb{R}_+, \text{ with coordinates } x \in (0, L) \text{ and } Y \in \mathbb{R}_+.  
\end{align*}

\noindent The equations for the velocity field $\bold{u}^{NS} := [u^{NS}, v^{NS}]$ and pressure $P^{NS}$ read 
\begin{align*}
&\bold{u}^{NS} \cdot \nabla \bold{u}^{NS} + \nabla P^{NS} = \eps \Delta \bold{u}^{NS}, \hspace{3 mm} \text{div}(\bold{u}^{NS}) = 0 \text{ on } \Omega
\end{align*}

\noindent with the following boundary conditions 
\begin{align*}
&\bold{u}^{NS}|_{Y = 0} = 0, \text{ (no-slip condition)}, \\
&\bold{u}^{NS}|_{Y \uparrow \infty} = [u^0_e(Y), 0], \text{ (convergence to an Euler flow) }, \\
&\bold{u}^{NS}|_{x = 0, L} = \text{ in-flow and out-flow provided through asymptotic expansion.}
\end{align*}

\noindent The shear Euler profile $u^0_e(Y)$ is given data. The in-flow and out-flow data are prescribed individually through the expansion (\ref{exp.u}). 

A fundamental question in fluid mechanics is to characterize the limit of $\bold{u}^{NS}$ as $\eps \downarrow 0$. Due to the mismatch of the no-slip boundary condition satisfied by $\bold{u}^{NS}$ for $\eps > 0$ and the typical no-penetration condition satisfied by a generic Euler flow, (which is $v^0_e|_{Y = 0} = 0$), one cannot expect $L^\infty$ convergence of $\bold{u}^{NS}$ to an Euler flow $[u^0_e(Y), 0]$. Instead, one expects convergence to what is now known as the Prandtl boundary layer. 

The first step in quantifying the asymptotic in $\eps$ behavior of $\bold{u}^{NS}$ is to rescale the normal variable, $Y$, and the normal component of the velocity, $v^{NS}$ according to 
\begin{align*}
&(x,y) = (x, \frac{Y}{\sqrt{\eps}}), \\
&[U^\eps(x,y), V^\eps(x,y), P^\eps(x,y)] = [U^{NS}(x,Y), \frac{v^{NS}(x,Y)}{\sqrt{\eps}}, P^{NS}(x,Y)].  
\end{align*}

\noindent The original, unscaled, variables $(x,Y)$ are known as ``Euler variables", whereas the new, scaled variables $(x,y)$ are known as ``Prandtl variables". 

One subsequently asymptotically expands $[U^\eps, V^\eps, P^\eps]$ in the following manner: 
\begin{align}
\begin{aligned} \label{exp.u}
&U^\eps = u^0_e + u^0_p + \sum_{i = 1}^n \sqrt{\eps}^i (u^i_e + u^i_p) + \eps^{N_0} u^{(\eps)} := u_s + \eps^{N_0} u^{(\eps)}, \\
&V^\eps = v^0_p + v^1_e + \sum_{i = 1}^{n-1} \sqrt{\eps}^i (v^i_p + v^{i+1}_e) + \sqrt{\eps}^n v^n_p + \eps^{N_0} v^{(\eps)} := v_s + \eps^{N_0} v^{(\eps)}, \\
&P^\eps = P^0_e + P^0_p + \sum_{i = 1}^n \sqrt{\eps}^i (P^i_e + P^i_p) + \eps^{N_0} P^{(\eps)} := P_s + \eps^{N_0} P^{(\eps)},
\end{aligned}
\end{align}

\noindent where the coefficients are independent of $\eps$. For our analysis, we will take $n = 4$ and $N_0 = 1+$ (this is motivated by \cite{GI}). First, the Euler flow $[u^0_e(Y), 0]$ is a given shear flow satisfying the assumptions delineated in Theorem \ref{thm.construct}. Next, $[u^i_e, v^i_e]$ are Euler correctors, which satisfy elliptic equations in the Euler variables $(x,Y)$. The terms $[u^i_p, v^i_p]$ are Prandtl correctors, which satisfy parabolic equations in Prandtl variables, $(x,y)$. $x$ behaves as a time-like variable whereas $y$ behaves as a space-like variable. 

Let us also introduce the following notation: 
\begin{align}
\begin{aligned} \label{intro.bar.prof}
 \bar{u}^i_p := u^i_p - u^i_p|_{y = 0}, \hspace{3 mm} \bar{v}^i_p := v^i_p - v^1_p|_{y = 0}, \hspace{3 mm} \bar{v}^i_e := v^i_e - v^i_e|_{Y = 0}.
\end{aligned}
\end{align}

\noindent The profile $\bar{u}^0_p, \bar{v}^0_p$ from (\ref{intro.bar.prof}) is classically known as the ``boundary layer"; one sees from (\ref{exp.u}) that it is the leading order approximation to the Navier-Stokes velocity, $U^\eps$. We sometimes use the notation $[\bar{u}, \bar{v}] = [\bar{u}^0_p, \bar{v}^0_p]$ due to the distinguished nature of the leading order boundary layer. Indeed, the quantity $[\bar{u}^0_p, \bar{v}^0_p]$ will solve the Prandtl equations, \eqref{Pr.leading}. As a matter of notation, we adopt the convention that for a given function, $f(x,y)$, the quantity $\bar{f} := f - f(x,0)$. 

The final layer, 
\begin{align*}
[u^{(\eps)}, v^{(\eps)}, P^{(\eps)}] = [\bold{u}^{(\eps)}, P^{(\eps)}].
\end{align*}

\noindent are called the ``remainders" and importantly, they depend on $\eps$. 
 
An important first step in understanding the asymptotic behavior of $(U^\eps, V^\eps, P^\eps)$ is to understand the approximate terms in the expansion (\ref{exp.u}), that is, all the terms aside from the remainders. One outcome of the present paper is that our analysis to establish Theorems \ref{thm.main}, \ref{thm:LP} will also help us to construct each term in the approximate solution and prove regularity estimates on these terms. Informally, our result below in this direction, Theorem \ref{thm.construct}, says given reasonably well-behaved boundary data on the sides $x = 0, L$, each of the terms in $[u_s, v_s, P_s]$ (see (\ref{exp.u})) can be constructed and are sufficiently regular. In order to state such a result precisely, we must introduce the equations satisfied by each term in $[u_s, v_s, P_s]$.

We now list the equations to be satisfied by the sub-leading order terms from (\ref{exp.u}). For $1 \le i \le n$, the $i$'th Euler layer satisfies the following system:
\begin{align} \label{des.eul.1.intro}
\left.
\begin{aligned}
&u^0_e \p_x u^i_e + \p_Y u^0_e v^i_e + \p_x P^i_e =: f^i_{E,1}, \\
&u^0_e \p_x v^i_e + \p_Y P^i_e  =: f^i_{E,2}, \\
&\p_x u^i_e + \p_Y v^i_e = 0, \\
&v^i_e|_{Y = 0} = - v^0_p|_{y = 0}, \hspace{5 mm} v^i_e|_{x = 0, L} = V_{E, \{0, L\}}^i \hspace{5 mm} u^i_e|_{x = 0} = U^i_{E}.
\end{aligned}
\right\}
\end{align}

\noindent In this case, since equation (\ref{des.eul.1.intro}) is elliptic in $(x,Y)$ we provide boundary data on both sides, $x = 0, L$. The given data for this problem is therefore the three functions $U^i_E(Y), V^i_{E,0}(Y), V^i_{E,L}(Y)$. The forcing terms $f^i_{E,1}, f^i_{E,2}$ are specifically given in Definition \ref{def.forcing}, and should be regarded as given for the purposes of stating the result. 

For $1 \le i \le n-1$, the $i$'th Prandtl layer satisfies 
\begin{align} \label{des.pr.1.intro}
\left.
\begin{aligned}
&\bar{u} \p_x u^i_p + u^i_p \p_x \bar{u} + \p_y \bar{u} [v^i_p - v^i_p|_{y = 0}] + \bar{v} \p_y u^i_p + \p_x P^i_p - \p_{yy} u^i_p := f^{(i)}, \\  
& \p_x u^i_p + \p_y v^i_p = 0,  \hspace{5 mm} \p_y P^i_p = 0\\  
& u^i_p|_{y = 0} = -u^i_e|_{y = 0}, \hspace{5 mm} [u^i_p, v^i_p]_{y \rightarrow \infty} = 0, \hspace{5 mm} v^i_p|_{x = 0} = V^i_P. 
\end{aligned}
\right\}
\end{align}

\noindent In this case, equation (\ref{des.pr.1.intro}) is parabolic, with $x$ controlling the evolution. As a result, the given data is the function $V^i_P(y)$. The forcing term, $f^{(i)}$, is defined in Definition \ref{def.forcing}. It should be regarded as given for the present discussion. For $i = n$, the difference is that $v^n_p$ will be defined so as to satisfy $v^n_p|_{y = 0} = 0$, unlike $v^i_p$ in (\ref{des.pr.1.intro}), and subsequently cut-off as $y \uparrow \infty$. The key point is that system \eqref{des.pr.1.intro} is of the form considered in \eqref{Orig.Orig.Pr.intro} and so we can apply Theorem \ref{thm:LP}.

\begin{theorem}[Construction of Approximate Navier-Stokes Solution] \label{thm.construct} Assume the shear flow $u^0_e(Y) \in C^\infty$, whose derivatives decay rapidly, and which is bounded above and below: $c_0 \le u^0_e \le C_0$ for some universal constants $0 < c_0 < C_0 < \infty$. Assume (\ref{OL.1}) regarding $\bar{U}^0_P$, and the conditions
\begin{align}  \label{compatibility.1.fin}
& \bar{V}^i_{Pyyy}(0) = \p_x g_1|_{x = 0, y = 0}, \\ \label{compatibility.2.fin}
&\bar{V}^i_P|'''(0) = \p_{xy}g_1|_{x = 0, y =0}, \\ \label{integral.cond.intro}
&\bar{U}^0_{P y}(0) U^{i}_{E,0}(0) - \int_0^\infty \bar{U}^0_{P} e^{-\int_1^y \bar{v}^0_{p}} \{f^{(i)}(y) - r^{(i)}(y) \} \ud y  = 0,
\end{align}

\noindent where $r^{(i)}(y) := \bar{V}^i_P \bar{U}^0_{P y} - \bar{U}^0_{P} \bar{V}^i_{Py}$, and $g^{(i)} = g^{(i)}(f^{(i)}, u^i_e|_{Y = 0})$ is an explicit quantity depending on the forcing $f^{(i)}$ and the boundary data $u^i_e|_{Y = 0}$, and is defined in (\ref{origPrLay}).  Let $V^i_{E,0}, V^i_{E,L}, U^i_E$  be prescribed smooth and rapidly decaying Euler data. We assume on the data standard elliptic  compatibility conditions at the corners $(0,0)$ and $(L,0)$. In addition, assume 
\begin{align}
&V^1_{E,0} \sim Y^{-m_1} \text{ or } e^{-m_1 Y} \text{ for some } 0 < m_1 < \infty,\\
& \|\p_Y^k \{ V^i_{E,0} - V^i_{E, L} \} \langle Y \rangle^M\|_\infty \lesssim L
\end{align}

\noindent Then, for $L << 1$, all profiles in $[u_s, v_s]$ exist and are smooth on $\Omega$. The following estimates hold: 
\begin{align}
\begin{aligned} \label{prof.pick}
&\bar{u}^0_p > 0,\bar{u}^0_{py}|_{y = 0} > 0, \bar{u}^0_{p yy}|_{y = 0} = \bar{u}^0_{p yyy}|_{y = 0} = 0 \\
&\| \nabla^K \{ u^0_p, v^0_p\} e^{My} \|_\infty \lesssim 1 \text{ for any } K \ge 0, \\
&\|u^i_p, v^i_p \|_\infty + \| \nabla^K  u^i_p  e^{My} \|_\infty + \| \nabla^K v^i_p e^{My} \|_\infty \lesssim 1 \text{ for any } K \ge 1, M \ge 0, \\
&\| \nabla^K \{u^1_e, v^1_e\}  w_{m_1} \|_\infty \lesssim 1 \text{ for some fixed } m_1 > 1  \\
&\| \nabla^K \{u^i_e, v^i_e\} w_{m_i}\|_\infty \lesssim 1 \text{ for some fixed } m_i > 1,
\end{aligned}
\end{align}

\noindent where $w_{m_i} \sim e^{m_i Y}$ or $(1+Y)^{m_i}$. 

In addition the following estimate on the remainder forcing holds: 
\begin{align} \label{thm.force.maz}
\| F_R|_{x = 0} w_0 \|_{L^2} + \| \p_x F_R \frac{w_0}{\sqrt{\eps}} \|_{L^2}  \lesssim \sqrt{\eps}^{n-1-2N_0},
\end{align}
\noindent where $F_R$ is the quantity defined in (\ref{forcingdefn}).
\end{theorem}

These constructions and regularity theory of the approximate solution to Navier-Stokes are of intrinsic interest. We also remark that another motivation for establishing our result is that these leading order constructions are important from the point of view of applications to the validity theory. Indeed, the estimates that we establish in this paper are in use in the works \cite{GI}, \cite{GIb}. 

\subsection{Notations}

We adopt several notation conventions throughout the paper. First, for functions $f(x,y)$, we define the norm $\|f \| := \| f \|_{L^2} = (\int_0^L \int_0^\infty f(x, y)^2 \ud y \ud x)^{\frac 1 2}$. Next, we define the antiderivative via $I_x[f]:= \int_0^x f(x',y) \ud x'$. 

\subsection{Main Ideas}

The key ingredient in the analysis of \eqref{origPrLay.beta} is a quotient estimate for $q_x$. We illustrate the main estimates for the simplified equation:  
\begin{align} \label{eq:intro:model}
- \p_{xy} \{ \bar{u}^2 q_y \} + v_{yyyy} = 0. 
\end{align}

By taking $\p_x$ of the main part of \eqref{origPrLay.beta}, taking the inner-product with $q_x$ (Lemma \ref{energy:estImate:1089}), and rearranging the main contributions, we obtain (upon omitting at each line below easy to estimate terms)
\begin{align} \n
&- (\p_{xy} \{ \bar{u}^2 q_{xy} \}, q_x) + (v_{xyyyy}, q_x) \\ \n
\sim & \frac{\p_x}{2} \int \bar{u}^2 q_{xy}^2 + (v_{xyy}, q_{xy})_{y = 0} + (v_{xyy}, q_{xyy}) \\ \n
\sim & \frac{\p_x}{2} \int \bar{u}^2 q_{xy}^2 + (\bar{u} q_{xyy} + 2 \bar{u}_y q_{xy} + \bar{u}_{yy}q, q_{xy})_{y = 0} \\ \n
& + (\bar{u} q_{xyy} + 2 \bar{u}_y q_{xy} + 2 \bar{u}_{xy} q_y, q_{xyy}) \\ \n
\sim & \frac{\p_x}{2} \int \bar{u}^2 q_{xy}^2 + \int \bar{u} q_{xyy}^2  + (2 \bar{u}_y q_{xy}, q_{xy})_{y = 0}  - ( \bar{u}_y q_{xy}, q_{xy})_{y = 0} \\ \n
&  - (\bar{u}_{xy} q_y, q_{xyy}) \\ \n
\sim & \frac{\p_x}{2} \int \bar{u}^2 q_{xy}^2 + \int \bar{u} q_{xyy}^2 + (\bar{u}_y q_{xy}, q_{xy})_{y = 0} + (\bar{u}_{xyy}q_y, q_{xy}) \\ \label{positive:energy:intro}
& + (\bar{u}_{xy} q_{yy}, q_{xy}) - (\bar{u}_{xy} q_y, q_{xy})_{y = 0}
\end{align}

\noindent The key positive boundary contribution $(\bar{u}_y q_{xy}, q_{xy})_{y = 0}$ holds naturally for Prandtl solutions as $\bar{u}_y|_{y = 0} > 0$. 

Next, for the term $(\bar{u}_{xy} q_{yy}, q_{xy})$ the key point is that $\bar{u}_{xy}|_{y = 0} \neq 0$. It is thus important to avoid the degeneracy at $y = 0$ in the term $\| \bar{u} q_{xy} \|$ by invoking a higher order norm term: 
\begin{align}
\| q_{yy} \| \lesssim \| v_{yyy} \| \text{ for } y << 1.
\end{align}

\noindent In order to close, we must further estimate $v_{yyyy}$ from the equation \eqref{eq:intro:model}, which gives (Lemma \ref{locvyyyy}
\begin{align}
\| v_{yyyy} \| \le \| \p_{xy} \{ \bar{u}^2 q_y \} \|.
\end{align}
Finally, we establish decay in $y$ in Lemma \ref{Hkestimate:8338}. 

We introduce 
\begin{align} \label{def:bar:u:theta}
\bar{u}^\theta := \bar{u} + \theta > 0
\end{align}

\noindent in our construction of approximate solutions and regard LDP \eqref{origPrLay.beta} as an initial value problem for given $v|_{x = 0}$ or $q|_{x = 0}$. We finally recover $u^0 = u|_{x = 0}$ and solve the original system, \eqref{Orig.Orig.Pr.intro}, via a necessary integrability condition \eqref{integral.cond.intro}. 

 The first three positive quantities in \eqref{positive:energy:intro} form our basic energy norm, $\| \cdot \|_{\mathcal{E}}$ in \eqref{norm.X.layer}. The remainder of our analysis is centered around propagating control over a slightly stronger norm, $X$, and upgrading to higher $\p_x$ derivatives. We note here that similar quotient estimates to those developed in this paper have played a role in establishing the validity of the Prandtl layer expansion, see for instance \cite{GI}, \cite{GIb}, \cite{Iyer-Global}.

\section{Regularity for the Prandtl and linearized Prandtl equations}

Our objective in this section is to establish Theorems \ref{thm.main} and \ref{thm:LP}, both of which will follow from the study of the system \eqref{Orig.Orig.Pr.intro}. 

\subsection{Linearized Derivative Prandtl Equation (LDP)}

We homogenize the system \eqref{Orig.Orig.Pr.intro} to remove the boundary condition $u|_{x = 0} = u_0(x)$ via: 
\begin{align} \label{antiPsi}
u = u_p - u_0(x) \psi(y), \hspace{3 mm} v = v_p + u_{0x}(x) I_\psi(y), \hspace{3 mm} I_\psi(y) := \int_y^\infty \psi(\theta) \ud \theta. 
\end{align}

\noindent  Here, we select $\psi$ to be a $C^\infty$ function satisfying the following: 
\begin{align} \label{psi.spec}
\psi(0) = 1, \hspace{3 mm} \int_0^\infty \psi = 0, \hspace{3 mm} \psi \text{ decays as } y \uparrow \infty, \hspace{3 mm} \psi^{(k)}(0) = 0 \text{ for } k \ge 1. 
\end{align}

\noindent  According to \eqref{Orig.Orig.Pr.intro}, the homogenized unknowns $[u,v]$ satisfy the system: 
\begin{align}
\begin{aligned} \label{origPrLay}
&\bar{u} \p_x u + u \p_x \bar{u} + \bar{v} \p_y u + v \p_y \bar{u} - \p_{yy} u = : g_1, \\
&u_x + v_y = 0, \\
&[u, v]|_{y = 0} = 0, \hspace{3 mm} \lim_{y \rightarrow \infty} u(x, y) = 0, \hspace{3 mm} v|_{x = 0} = \bar{V}_0(y),
\end{aligned}
\end{align}
where the following quantities have been introduced: 
\begin{align}
\begin{aligned}
&g_1 := f + G, \\
&G := - \bar{u} \psi u_{0x} - \bar{u}_{x} \psi u_0 - \bar{v} \psi' u_0 - \bar{u}_{y} u_{0x} I_\psi + \psi'' u_0, \\
&\bar{V}_0(y) := V_0 -u_0(0) I_\psi(y).
\end{aligned}
\end{align}

\noindent  By applying $\p_y$, we obtain the system: 
\begin{align}  \label{eval.2}
- \bar{u} v_{yy} + v \bar{u}_{yy} - u \bar{v}_{yy} + \bar{v} u_{yy} - u_{yyy} = \p_y g_1.
\end{align}

\noindent We rewrite the Rayleigh operator as 
\begin{align}
- \bar{u} v_{yy} + v \bar{u}_{yy} = - \p_y \{ \bar{u}^2 q_y \}, \hspace{3 mm} q := \frac{v}{\bar{u}}.
\end{align}

\noindent  By further taking $\p_x$ we derive the following linear Derivative Prandtl Equation (LDP) for the quotient, $q$, which is the main focus of our paper:
\begin{align}
\begin{aligned} \label{origPrLay.beta}
&- \p_{xy} \{ \bar{u}^2 q_y \} + \p_y^4 v + \kappa \Lambda + \kappa U   = \p_{xy} g_1, \\
&q|_{y = 0} = 0, \hspace{3 mm} q|_{x = 0} = \frac{1}{\bar{u}}|_{x= 0}(y) \bar{V}_0(y) := f_0(y). 
\end{aligned}
\end{align}

\noindent We have introduced the artificial parameter $\kappa$ above. The system of interest is $\kappa = 1$. The reason for including this artificial parameter is, in Proposition \ref{Prop.1} below, we also treat the $\kappa = 0$ case of nonlinear Prandtl, \eqref{Pr.leading}. We have defined the following operators: 
\begin{align} \label{def:Lambda}
&\Lambda(v) := \bar{v}_{xyy}I_x[v_y] + \bar{v}_{yy} v_y - \bar{v}_{x}I_x[v_{yyy}] - \bar{v} v_{yyy}, \\ \label{def:op:U}
&U(u^0) := - \bar{v}_{xyy}u^0 + \bar{v}_{x}u^0_{yy}.
\end{align}

We now state our basic proposition for the homogenized system. To do so, we need to specify norms in which we measure the quantity, $q$, as well as the forcing, $f$. Let $\chi$ denote the following cut-off function: 
\begin{align}\label{basic.cutoff}
\chi(y) = \begin{cases} 1 \text{ on } 0 \le y \le 1 \\ 0 \text{ on } y \ge 2\end{cases} \hspace{5 mm} \chi'(y) \le 0 \text{ for all }y. 
\end{align}
 Fix $w = e^{Ny}$ for the $N$ from Theorem \ref{thm.main}. Denote by $q^{(k)} := \p_x^k q$. We will now define several norms: 
\begin{align} 
\begin{aligned} \label{norm.X.layer}
&\| q \|_X :=  \sup_{0 \le x \le L} \Big( \|\bar{u} q_{xy} \|_{L^2_y} + \|q_{yyy} w (1 - \chi) \|_{L^2_y} \Big)  + \| \sqrt{\bar{u}} q_{xyy} w \| + \| v_{yyyy} w \| \\ 
& \qquad \qquad + \| \sqrt{\bar{u}_y} q_{xy} \|_{y = 0}, \\
&\| q \|_{\mathcal{E}} := \sup_{0 \le x \le L} \| \bar{u} q_{xy} \|_{L^2_y} + \| \sqrt{\bar{u}} q_{xyy} \| + \| \sqrt{\bar{u}_y} q_{xy} \|_{y = 0}\\
&\| q \|_{\mathcal{H}} := \sup_{0 \le x \le L} \| q_{yyy} w\{1 - \chi \}\|_{L^2_y} + \| v_{yyyy} w \{1 - \chi\} \| + \| q_{xyy} w \{1 - \chi\} \| \\
& \| q \|_{X_k} := \| q^{(k)} \|_X,  \hspace{3 mm} \| q \|_{\mathcal{E}_k} := \| q^{(k)} \|_{\mathcal{E}}, \hspace{3 mm} \| q \|_{\mathcal{H}_k} := \| q^{(k)} \|_{\mathcal{H}}, \\
& \| q \|_{X_{\langle k \rangle}} = \sum_{i = 0}^k \| q \|_{X_{i}}.
\end{aligned}
\end{align}

We also define the following norm in which we control the forcing term: 
\begin{align}
\| f \|_{\tilde{X}_{\langle k \rangle}} = \| e^{Ny} \p_x^{k} \p_{xy}f  \| + \| \langle y \rangle \p_x \p_x^k \p_{xy} f \|. 
\end{align}

\begin{proposition} \label{thm.diff} Let the initial data, $V_0(y)$, the forcing, $f$, and the boundary data, $u_0(x)$,  satisfy the compatibility conditions (\ref{compatibility.1}) and (\ref{compatibility.2}). Assume also the integral condition, (\ref{integral.cond}). Then there exists a unique solution to (\ref{Orig.Orig.Pr.intro}) satisfying the following estimate: 
\begin{align*}
\| q \|_{X_{\langle k \rangle}} \lesssim \| f \|_{\tilde{X}_{\langle k \rangle}} + C(V_0, u_0). 
\end{align*}
\end{proposition}

Theorem \ref{thm.main} is a corollary of Proposition \ref{thm.diff} upon differentiating the Prandtl equation, (\ref{Pr.leading}), in $x$. Theorem \ref{thm:LP} is also an immediate corollary of Proposition \ref{thm.diff}. 

\subsection{Compatibility Conditions and Reconstruction of $u^0$}  \label{appendix.prandtl}

It is possible and natural to solve \eqref{Orig.Orig.Pr.intro} for a given initial condition $u^0 = u|_{x = 0}$, instead of a given $v^0$. In such a formulation, no integrability condition, \eqref{integral.cond}, is needed, but $v^0$ must be determined from solving 
\begin{align}
-\bar{u} v_y + \bar{v} u^0_y + \bar{u}_y v + \bar{u}_x u^0 - u^0_{yy} = g_1 \text{ at } x= 0. 
\end{align}

\noindent In this case, the compatibility conditions are $\p_x^k v^0(0) = \p_x^k v^0_y(0) = 0$, where 
\begin{align}
&\p_x^k v^0 := \p_x^k \Big\{ - \bar{u} \int_0^y \frac{1}{\bar{u}} \{ \bar{v} u^0_y - u^0 \bar{u}_x + u^0_{yy} + g_1 \} \Big\}, \\
&\p_x^k v^0_y := \p_x^k \Big\{ - \bar{u} \int_0^y \frac{1}{\bar{u}} \{ \bar{v} u^0_y - u^0 \bar{u}_x + u^0_{yy} + g_1 \} \Big\}_y
\end{align}

However, in our $v$-formulation, it is interesting to reconstruct $u^0$ via a given initial condition $V_0$. The purpose of this subsection is to reconstruct $u^0$ from $V_0$ under a necessary integrability condition on $V_0$, \eqref{integral.cond}. 

Our aim now is to derive compatibility conditions for the initial data. By computing $\p_x$ of (\ref{origPrLay}) and evaluating at $y = 0$, we obtain the condition: 
\begin{align*}
v_{yyy}|_{y = 0}(x) = \p_x g_1|_{y = 0}(x), \qquad x \in (0,L). 
\end{align*}
The first order compatibility condition that arises from this entails matching now the initial datum evaluated at $y = 0$ with the boundary datum evaluated at $x = 0$ in the standard manner for initial-boundary value problems: 
\begin{align}
V_0'''(0) = \bar{V}_0'''(0) = \p_x g_1|_{y = 0}(0). 
\end{align}
The first equality above follows from the property of $\psi'''(0) = 0$, whereas the second equality above is the natural outcome of the aforementioned matching process.  

We also require the second-order compatibility which can be obtained as follows. Taking $\p_x$ of (\ref{eval.2}): 
\begin{align*}
- \p_x \{ - \bar{u} v_{yy}+ v \bar{u}_{yy} + \bar{v} u_{yy} - u \bar{v}_{yy} \} + \p_y^4 v = \p_{xy}g_1. 
\end{align*}

\noindent  Evaluating at $y = 0$ gives the identity: 
\begin{align*}
\p_y^4 v|_{y = 0}(x) = \p_{xy} g_1|_{y = 0}(x) \text{ for } x \in (0,L).
\end{align*}
This then gives our second order compatibility condition of 
\begin{align}
\p_y^4 V_0(0) = \p_y^4 \bar{V}_0(0) = \p_{xy}g_1|_{y = 0}(0). 
\end{align}

This derivation motivates our definition of compatibility conditions, which we state formally now.
\begin{definition}[Compatibility at Corners] \label{def:compatibility:condition} The first and second order compatibility conditions are given by 
\begin{align}
\p_y^3 V_0(0)  = \p_x g_1|_{y = 0}(0), \qquad \p_y^4 V_0(0) = \p_{xy}g_1|_{y = 0}(0). 
\end{align}
Higher order compatibility conditions can be derived in the same manner, and are stated implicitly as follows
\begin{align} \label{compatibility.1}
& \p_y^3 (\p_x^k v|_{x = 0})(0) = \p_x^{k+1} g_1|_{y = 0}(0) \\ \label{compatibility.2}
&\p_y^4 (\p_x^k v|_{x = 0})(0)  = \p_x^{k+1} \p_y g_1|_{y =0}(0). 
\end{align} 
\end{definition}

\noindent Starting from the $q$ formulation in (\ref{origPrLay.beta}), we will further distribute on the Rayleigh term: 
\begin{align*}
- \p_y \{ \bar{u}^2 q_{xy} \} - \p_y \{2 \bar{u} \bar{u}_{x} q_y \} + \p_y \p_x \{ \bar{v} u_y - u\bar{v}_{y} \} + \p_y \p_y^3 v = \p_y \p_x g_1. 
\end{align*}

\noindent  We now compute at $\{x = 0\}$:
\begin{align} \n
\bar{u}^2 q_{xy} = &- \int_y^\infty \p_y\{ \bar{u}^2 q_{xy} \} \ud y' \\  \n
= & \int_y^\infty \p_y \Big\{ \p_x g_1 - \p_y^3 v + 2 \bar{u} \bar{u}_{x} q_y - \bar{v}_{x} u^0_y + \bar{v}  v_{yy} - v_y \bar{v}_{y} + u^0 \bar{v}_{xy}  \Big\} \ud y' \\ \label{bahumbug}
= & - \{ \p_x g_1 - \p_y^3 v + 2 \bar{u} \bar{u}_{x} q_y - \bar{v}_{x}u^0_y + \bar{v} v_{yy} - v_y \bar{v}_{y} + u^0 \bar{v}_{xy} \}. 
\end{align}
It is clear that all quantities are vanishing at $y = 0$. We thus have that $\bar{u} q_{xy}|_{x = 0} \in L^2(\mathbb{R}_+)$. A computation of $\p_y$ shows: 
\begin{align*}
&\p_{xy} g_1+ \p_y^4 v + \p_y \{ 2 \bar{u}_{x}\bar{u} q_y - \bar{v}_{x} u^0_y + \bar{v} v_{yy} - v_y \bar{v}_{y} + u^0 \bar{v}_{yy} \}|_{y = 0} \\
 =& \p_{xy}g_1(0,0) + \p_y^4v|_{x = 0}(0) = 0. 
\end{align*}

\noindent  Thus $q_{xy}$ itself is in $L^2$. Using this we may easily bootstrap to higher order in $\p_y$ compatibility conditions for $v^0$ which we refrain from writing. These conditions in turn assure that: 
\begin{lemma} \label{lemma.compat.1}   Assume the compatibility conditions on $V_0$ given in Definition \ref{def:compatibility:condition}. Assume exponential decay on $\p_y^k V_0$ for $k \ge 1$. Then there exist functions $f_k(y) \in L^2_w(\mathbb{R}_+) \cap C^\infty(\mathbb{R}_+)$ for exponential weight $w$ such that 
\begin{align} \label{whois}
\p_x^k q_y|_{x = 0} = f_k(y) \in L^2_w(\mathbb{R}_+) \text{ for } k \ge 1.
\end{align} 

\noindent  Moreover, $f_k$ depend only on the given profile $V_0$ and the forcing term $g_1$. 
\end{lemma}

Our task now is to establish criteria on the initial data, $V_0$ so that $u^0(y) := u_p|_{x = 0}$ can be bounded. We evaluate the velocity equation (\ref{Orig.Orig.Pr.intro}) at $x = 0$ to obtain the equation: 
\begin{align} \label{wknd}
&\mathcal{L}_{\parallel} u^0 = f(y) - r(y), \hspace{5 mm} u^0(0) = u_0(0), \qquad r(y) := V_0 \bar{u}_{y} - \bar{u} V_0',
\end{align}

\noindent where we have defined the operators
\begin{align*}
&\mathcal{L}_{\parallel} u^0 := - u^0_{yy} + \bar{v} u^0_{y} - u^0 \bar{v}_{y}, \qquad  \mathcal{L}_{\parallel, y} u^0 := - u^0_{yyy} + \bar{v} u^0_{yy} - u^0 \bar{v}_{yy},
\end{align*}

\noindent To invert equation \eqref{wknd} for $u^0(y)$, we assume the following condition on the datum $V_0(y)$: 
\begin{align} \n
-\bar{u}_{y}(0, 0) u_0(0) - \int_0^\infty \bar{u} e^{-\int_1^y \bar{v}} \{f(y) - r(y) \} \ud y  = 0.
\end{align}
This, then, motivates the following definition upon replacing notationally $\bar{u}(0, y) = \bar{U}_0(y)$: 
\begin{definition}[Compatibility Integral Condition] \label{def:integral:cond} We say the datum, $V_0(y)$ satisfies the ``compatibility integral condition" if 
\begin{align} \n
&\int_0^\infty \bar{U}_0 e^{- \int_1^y \bar{v}(0, z) \ud z} (V_0 \bar{U}_0' - \bar{U}_0 V_0') \ud y \\  \label{integral.cond}
&= \bar{U}_0'(0) u_0(0) + \int_0^\infty \bar{U}_0 e^{- \int_1^y \bar{v}(0, z) \ud z} f(y) \ud y. 
\end{align}
\end{definition}

\begin{lemma} \label{propn.span}  Elements of the three dimensional kernel of $\mathcal{L}_{\parallel, y}$ can be written as the following linear combination: $c_1 \bar{u} + c_2 \tilde{u}_s + c u_R$,  where $c_1, c_2, c \in \mathbb{R}$. Here: 
\begin{align*}
&\tilde{u}_s := \bar{u} \int_1^y \frac{|\bar{u}(1)|^2}{|\bar{u}|^2} \exp \Big[ \int_1^z \bar{v} \ud w \Big] \ud z, \\
&u_R :=  \tilde{u}_s \int_0^y \bar{u} \exp \Big[ -\int_1^z \bar{v} \Big]  -\bar{u} \int_0^y \tilde{u}_s \exp \Big[ - \int_1^z \bar{v} \Big].
\end{align*}
\end{lemma}
\begin{proof} Note $\mathcal{L}_{\parallel, y} u = 0$ if and only if $\mathcal{L}_{\parallel} u = c$ for a constant $c$. One solution to the homogeneous equation, $\mathcal{L}_{\parallel} u = 0$, is $\bar{u}$. By supposing the second spanning solution is of the form $\tilde{u}_s := \bar{u} a(y)$, we may derive the equation: $a''(y) = \Big[ \bar{v} - 2\frac{\bar{u}_{y}}{\bar{u}} \Big] a'(y)$. Solving this equation gives one solution:
\begin{align} \label{a.def}
a'(y) = \frac{|\bar{u}(1)|^2}{|\bar{u}|^2} \exp \Big[ \int_1^y \bar{v} \Big], \hspace{3 mm} a(y) = \int_1^y \frac{|\bar{u}(1)|^2}{|\bar{u}|^2} \exp \Big[ \int_1^z \bar{v} \ud w\Big] \ud z. 
\end{align}
 
\end{proof}

We shall need asymptotic information about $\tilde{u}_s$:
\begin{lemma} \label{lemma.asy.1} \normalfont As defined in Lemma \ref{propn.span}, $\tilde{u}_s$ satisfies the following asymptotics: 
\begin{align} 
\begin{aligned}
&\tilde{u}_s|_{y = 0} \sim -1 \text{ and }\tilde{u}_{sy} \sim 1 \text{ as } y \downarrow 0, \\
& \tilde{u}_{sy},  \tilde{u}_{sy}, \tilde{u}_s  \sim \exp[\bar{v}(\infty) y] \text{ as } y \uparrow \infty. 
\end{aligned}
\end{align}
\end{lemma}
\begin{proof} For convenience, denote $g(y) = \exp ( \int_1^y \bar{v} )$. By rewriting $\bar{v} = \bar{v}(\infty) + (\bar{v} - \bar{v}(\infty))$, and using that the latter difference decays rapidly, we obtain the basic asymptotics $g \sim \exp( \bar{v}(\infty) y)$ as $y \uparrow \infty$. An expansion of $a$, given in (\ref{a.def}), near $y = 0$ gives $a(y) \approx \int_1^y \frac{1}{z^2} \ud z \sim -\frac{1}{z}|_{1}^y  = 1-\frac{1}{y}$. Thus: $\tilde{u}_s|_{y = 0} \sim \bar{u} (1 - \frac{1}{y} ) \sim -1$. At $y = \infty$, we have the asymptotics:
\begin{align*}
 \tilde{u}_s = \bar{u} \int_1^y \frac{|\bar{u}(1)|^2}{|\bar{u}|^2} g(z)  \sim \int_1^y \exp[\bar{v}(\infty)  z] \ud z \sim \exp[\bar{v}(\infty) y]. 
\end{align*}

\noindent  We now differentiate to obtain 
\begin{align*}
\tilde{u}_{sy} = \bar{u}_{y}a + \bar{u} a'(y) = \bar{u}_{y}a(y) + \frac{|\bar{u} (1)|^2}{\bar{u}} g(y) \sim \exp[\bar{v}(\infty) y].
\end{align*}

\noindent  To evaluate $\tilde{u}_{sy}$ at $y = 0$, we need more precision. Expansions give: 
\begin{align*}
&\bar{u} a'(y) = \frac{|\bar{u}(1)|^2}{\bar{u}} g(y) \sim \frac{|\bar{u}(1)|^2}{\bar{u}_{y}(0)y} g(y) \text{ for }y \sim 0, \text{ and } \\
&\bar{u}_{ y}a(y) \sim \bar{u}_{y}(0) |\bar{u}(1)|^2 \int_1^y \frac{1}{|\bar{u}|^2} g(z)  \sim \frac{|\bar{u}(1)|^2}{\bar{u}_{y}(0)} \int_1^y \frac{g(z)}{z^2} \ud z. 
\end{align*}

\noindent  We have used the fact that $\frac{1}{|\bar{u}|^2}$ does not contribute a factor of $\frac{1}{z}$ following the singularity of $\frac{1}{z^2}$. Indeed, Taylor expanding, using that $\bar{u}^0_{ yy}(0) = \bar{u}_{yyy}(0) = 0$ (see the first identity in (\ref{prof.pick})), and the elementary identity for any $a, b \in \mathbb{R}$, $\frac{1}{a-b} - \frac{1}{a} = \frac{b}{a(a-b)}$, one obtains: 
\begin{align*}
\frac{1}{|\bar{u}(z)|^2} =  \frac{1}{\bar{u}_{y}(0)^2 z^2} + \bigO(z). 
\end{align*}

It remains to show $\int_1^y \frac{g(z)}{z^2} \ud z \sim -\frac{g(y)}{y}$. We decompose the integral into region $[1, y_\ast]$ and $[y_\ast, y]$ for $0 < y \le z \le y_\ast$. The $[1, y_\ast]$ integral contributes an $\bigO(1)$ constant. In the $[y_\ast, y]$  region, the Taylor expansion is valid: 
\begin{align*}
\int_{y_\ast}^y \frac{g(z)}{z^2} \ud z \sim \int_{y_\ast}^y \frac{g(y)}{z^2} \ud z + g'(y)\int_{y_\ast}^y \frac{z-y}{z^2} \ud z \sim g(y) (\frac{1}{y_\ast} - \frac{1}{y} ) + g'(y) \phi(y), 
\end{align*}

\noindent  where $| \phi(y)| \lesssim |\log y|$. We now use that $\bar{v}(0) = \bar{v}_{p}(0) = 0$ and $g'(y) = \bar{v}(y) g(y)$ to show that $g'(y) \sim y^2$. Thus, $g'(y) \phi(y)$ vanishes as $y \rightarrow 0$. We thus have verified that $I(y) \sim -\frac{g(y)}{y}$. 

We now compute two derivatives: 
\begin{align*}
\tilde{u}_{syy} &= \bar{u}_{yy}a + 2\bar{u}_{y}a'(y) + \bar{u} a''(y) \\
& \sim a''(y) \sim \p_{y} \{ \frac{1}{|\bar{u}|^2} \exp[\bar{v}(\infty) y] \} \sim \exp[\bar{v}(\infty) y] \text{ as } y \uparrow \infty.
\end{align*}
\end{proof}

\begin{lemma}   Assume the integral condition, (\ref{integral.cond}) is satisfied by the initial data $v_0$. Then the solution $u^0$ to (\ref{wknd}) exists and satisfies: 
\begin{align} \label{base1}
&|\p_y^k u^0 e^{My}|_\infty \le C_{K,M} (V^0, g_1) \text{ for }k \ge 1, \\
&u^0(0) = u_0(0) \text{ and }  \lim_{y \uparrow \infty} u^0(y) = 0. 
\end{align}
\end{lemma}
\begin{proof} First, we compute the Wronskian of $\bar{u}^0_{p}$ and $\tilde{u}_s$: 
\begin{align*}
W = \bar{u} \tilde{u}_{sy} - \tilde{u}_s \bar{u}_{y} = \bar{u}(1)^2 \exp \Big[ \int_1^y \bar{v} \Big].
\end{align*}

Next, we express the solution to (\ref{wknd}) in the following manner: 
\begin{align} \n
u^0 = &u_0(0) \frac{\tilde{u}_s}{\tilde{u}_s(0)} + c_1 \bar{u} - \frac{1}{|\bar{u}(1)|^2} \tilde{u}_s \int_0^y \bar{u} e^{- \int_1^z \bar{v}} (f(z) - r(z) ) \ud z \\ \label{yanks}
& + \frac{1}{\bar{u}(1)^2} \bar{u} \int_0^y \tilde{u}_s e^{-\int_1^z \bar{v}} ( f(z) - r(z) ) \ud z. 
\end{align}

We now compute $\tilde{u}_s(0) = - \frac{|\bar{u}(1)|^2}{\bar{u}_{y}(0)} e^{\int_1^0 \bar{v}}$. Using this, we now evaluate at $y = \infty$ and observe that the terms with a $\tilde{u}_s$ prefactor vanish according to the integral condition, (\ref{integral.cond}).
\begin{align*}
-\frac{\bar{u}_{ y}(0)}{\bar{u}(1)^2} u_0(0) e^{-\int_1^0 \bar{v}} - \frac{1}{|\bar{u}(1)|^2} \int_0^\infty \bar{u} e^{- \int_1^y \bar{v}} (f - r(y)) \ud y = 0. 
\end{align*} 

This proves that $u^0$ as defined in (\ref{yanks}) is bounded as $y \uparrow \infty$. We next notice that the derivative of $\int_0^y \bar{u} e^{-\int_1^\infty \bar{v}} (f - r(z) )$ is the integrand itself, which decays fast enough to eliminate $\tilde{u}_s$ at $\infty$. Therefore we also see that $\p_y^k u^0$ for $k \ge 1$ decays rapidly. 

Finally, we need to ensure that $u^0 \rightarrow 0$ as $y \uparrow \infty$. It is clear that $\mathcal{L}_{\parallel, y} u^0 = 0$, and so we are free to modify $u^0$ by factors of $\bar{u}$. Thus we modify (\ref{yanks}) by subtracting off a factor of $c \bar{u}$, for $c$ appropriately selected so as to ensure $u^0(\infty) = 0$. This concludes the proof. 
\end{proof}

Summarizing the above, 
\begin{lemma}   Assume smooth data, $V_0$, are prescribed that satisfies the compatibility conditions (\ref{compatibility.1}), (\ref{compatibility.2}), as well as higher order compatibility conditions. Assume also that $V_0$ satisfies the integral condition (\ref{integral.cond}). Let $q = \frac{v}{\bar{u}}$ solve (\ref{origPrLay.beta}) and $u^0$ be constructed from $v$ via (\ref{wknd}). Then $[u = u^0 - \int_0^x v_y, v]$ solve (\ref{origPrLay}). Further, let $[u_p, v_p]$ be reconstructed from $[u, v]$ using (\ref{antiPsi}). Then $[u_p, v_p]$ are solutions to (\ref{Orig.Orig.Pr.intro}). Moreover, the initial datum $u^0 = u|_{x = 0}$ satisfies the decay estimates \eqref{base1}.
\end{lemma}


\subsection{Construction of Approximate Solution}

First, we recall the definition $\bar{u}_\theta = \bar{u} + \theta$ from \eqref{def:bar:u:theta}. We also define 
\begin{align}
 q^{(\theta)} := \frac{v^{(\theta)}}{\bar{u}}. 
\end{align}

\begin{definition}[$\theta$-Approximant]
\begin{align} \label{theta:approx:1}
&- \p_{xy}(\bar{u}_\theta^2 q^{(\theta)}_y) + v^{(\theta)}_{yyyy} + \kappa \Lambda(v^{(\theta)}) + \kappa U(u^0) = F^{(\theta)}, \\
&q^{(\theta)}|_{y = 0} = 0, \qquad q^{(\theta)}|_{x = 0} = Q_\theta(y),  
\end{align}
where the forcing $F^{(\theta)} := \p_{xy} g_1$. We also introduce the notation $V_\theta(y) = \bar{u} Q_\theta(y)$.   
\end{definition}

\begin{definition}[$\theta$-Dependent Compatibility Conditions] We now define the compatibility required on the initial conditions we take, $Q_{\theta}(y)$. We assume that $u^{0}(y)$ is given. We need to enforce standard parabolic compatibility conditions on $Q_\theta(y)$ at $y = 0$. These can be read off from the equation \eqref{theta:approx:1} in the usual manner. For instance, for the first order compatibility condition, we obtain 
\begin{align}
Q_{\theta, 1}(y) := q^{(\theta)}_x(0, y) = \int_y^\infty \frac{\mathcal{G}_0}{\bar{u}_\theta^2} + 2 \frac{\bar{u}_{\theta x}}{\bar{u}_\theta} Q_{\theta}' - \frac{V_{\theta}'''}{\bar{u}_\theta^2} \ud y', 
\end{align} 
where
\begin{align}
\mathcal{G}_0(y) := (F^{(\theta)} - \kappa U(u^0) - \kappa \Lambda(v^{(\theta)} )|_{x = 0}
\end{align}
We thus need to ensure that $Q_{\theta, 1}(0) = 0$, which is the ``first order compatibility condition". Higher order compatibility conditions are obtained in a similar manner. 
\end{definition}

To construct solutions to \eqref{theta:approx:1}, we will introduce an artificial truncation at $y = N$ and then send $N \rightarrow \infty$. Moreover, will make $\bar{u}_\theta(y)$ constant for $y > \frac N 2$. To do this, we define $\bar{u}_\theta^{(N)} \in C^\infty$ satisfying the following properties  
\begin{align}
\bar{u}_{\theta}^{(N)}(x, y) = \begin{cases} \bar{u}_\theta(x, y) \text{ for } 0 \le y \le \frac N 4 \\ \bar{u}_\theta(x, \frac N 2) \text{ for } \frac N 2 \le y \le N  \end{cases}. 
\end{align}
We shall now consider the problem on $(x, y) \in (0, L) \times (0, N)$, 
\begin{align}
\begin{aligned} \label{fp:2:N}
&- \p_{y}((\bar{u}_{\theta}^{(N)})^2 q^{(\theta, N)}_{xy}) + v^{(\theta, N)}_{yyyy}+ \kappa \Lambda(v^{(\theta, N)}) + \kappa U(u^0) = F^{(\theta)}  \\
& q^{(\theta, N)}|_{y  = 0, N} = 0,  \qquad q^{(\theta, N)}_{yy}|_{y = 0,  N} = 0, \qquad q^{(\theta, N)}|_{x = 0} = f_0(y).
\end{aligned}
\end{align}
We fix a Fourier sine basis adapted to the interval $[0, N]$, by letting 
\begin{align}
e_i(y) := \sin( i \frac{2\pi}{N} y ).
\end{align}
For each $n < \infty$, we search for solutions, $q^{(n,N,\theta)} \in \text{span}\{e_i\}$ for $i = 1,..,n$, which satisfy, for $e_i, i = 1,...,n$, 
\begin{align} \n
&(\p_x ((\bar{u}_\theta^{(N)})^2 q^{(N, n,\theta)}_y ), e_i') + (\p_{yy} (\bar{u} q^{(N, n, \theta)} ), e_i'') + 2 \bar{u}_y q^{(N, n,\theta)}_y e_i'(0)\\  \n
&+ \kappa (\bar{v}_{xyy}I_x[v^{(N, n,\theta)}_y], e_i) + \kappa (\bar{v}_{yy} v^{(N, n,\theta)}_y, e_i) + \kappa (I_x[v^{(N, n,\theta)}_{yy}], \bar{v}_{xy}e_i + \bar{v}_x e_i')  \\  \label{weak.1}
&+ \kappa (v^{(N, n,\theta)}_{yy}, \bar{v}_y e_i + \bar{v} e_i')  + \kappa (U(u^0), e_i)  = (F^{(\theta)}, e_i)
\end{align}
We now propose $q^{(N, n,\theta)} = \sum_{l = 1}^n b^{(n)}_l(x) e_l$. The first claim is that there exist coefficients, $b_i^{(n)}(x)$ such that $q^{(n,\theta)}$ satisfies \eqref{weak.1}. Inserting this expansion into \eqref{weak.1}, we obtain 
\begin{align} \n
&(\p_x \{(\bar{u}_\theta^{(N)})^2 \sum_{l = 1}^n b^{(n)}_l(x) e_l' \}, e_i') + (\p_{yy} \{ \bar{u} \sum_{l = 1}^n b^{(n)}_l(x) e_l \}, e_i'')\\ \n
+& 2(\bar{u}_y \sum_{l = 1}^n b^{(n)}_l(x) e_l'(0), e_i'(0)) + \kappa \sum_{l = 1}^n I_x[b^{(n)}_l(x)] (\bar{v}_{xyy} \p_y (\bar{u} e_l), e_i) \\ \n
& + b^{(n)}_l(x) (\bar{v}_{yy} \p_y (\bar{u} e_l), e_i)  + \kappa ( I_x[ \p_{yy} ( \bar{u} \sum_{l = 1}^n b^{(n)}_l(x) e_l(y)) ], \bar{v}_{xy} e_i + \bar{v}_x e_i' ) \\ \n
& + \kappa (\p_{yy} ( \bar{u} \sum_{l = 1}^n b^{(n)}_l(x) e_l ), \bar{v}_y e_i + \bar{v} e_i') + (\kappa U(u^0), e_i) \\ \label{weak:2}
& - 2(\bar{u}_\theta^{(N)} \p_x \bar{u}_\theta^{(N)} \sum_{l = 1}^n b^{(n)}_l(x) e_l' , e_i')= (F^{(\theta)}, e_i). 
\end{align}
We thus obtain an $n \times n$ system of ODE's for $b_l^{(n)}$ $l = 1,...,n$, which can be solved using standard ODE theory. Indeed, for $i = 1,..,n$, we have the following system of ODEs:
\begin{align} \label{ODE:x:time}
\Gamma^{\ell}_i \dot{b}^{(n)}_\ell(x) + \sum_{l = 1}^{n} A^l_i(x) b_l^{(n)}(x) + \sum_{l = 1}^n K^l_i(x) I_x[b^{(n)}_l(x)] = F_i(x), \text{ for } i = 1,\dots,n
\end{align} 
where 
\begin{align}  \n
A^l_i(x) := & (\p_{yy}(\bar{u} e_l(y)), e_i''(y) ) + 2 \bar{u}_y(x,0) e_l'(0) e_i'(0) + \kappa (\bar{v}_{yy} \p_y (\bar{u} e_l), e_i) \\
& + \kappa ( \p_{yy}(\bar{u} e_l), \bar{v}_y e_i + \bar{v} e_i' ) , \\
K^l_i(x) := & \kappa (\bar{v}_{xyy} \p_y (\bar{u} e_l), e_i) + \kappa ( \p_{yy}(\bar{u} e_l), \bar{v}_{xy} e_i + \bar{v}_x e_i' ) \\
\Gamma^\ell_i := & (\bar{u}_\theta^2 e_\ell', e_i' ) \\
F_i(x) := & (F^{(\theta)}, e_i). 
\end{align}
This can be solved for each fixed $n \in \mathbb{N}$, as the coefficients $A^l_i(x), K^l_i(x),  F_i(x)$ are smooth. Moreover, the coefficients $\Gamma$ form a symmetric, positive definite matrix. To establish positive definiteness, we fix a vector $\vec{b} =( b_j)_{j = 1}^n$, Multiplying $\vec{b}^T \Gamma \vec{b} = \| \bar{u}_\theta \eta \|_{L^2_y}^2$, where $\eta := \sum_j b_j e_j'$. In turn, this implies that $\eta = 0$, which also implies that $b_j = 0$.  

We now would like to send $n, N \rightarrow \infty$, for which we need to rely upon uniform bounds. By linearity of \eqref{weak:2}, we can replace the $e_i(y)$ by $q^{(n, N, \theta)}$, which gives
\begin{align} \n
&\frac{\p_x}{2}( |\bar{u}_\theta^{(N)}|^2 q^{(n, N, \theta)}_y, q^{(n, N, \theta)}_y) + (\p_{yy}(\bar{u} q^{(n, N, \theta)}), q^{(n, N, \theta)}_{yy}) + 2 \bar{u}_y |q^{(n, N, \theta)}_y(0)|^2 \\ \n
&+ \kappa (\bar{v}_{xyy} I_x[v^{(n, N, \theta)}_y], q^{(n, N, \theta)}) + \kappa (\bar{v}_{yy} v^{(n, N, \theta)}_y, q^{(n, N, \theta)}) \\ \n
& + \kappa (I_x[v^{(n, N, \theta)}_{yy}], \bar{v}_{xy}q^{(n, N, \theta)} + \bar{v}_x q^{(n, N, \theta)}_y)  + \kappa (v^{(n, N, \theta)}_{yy}, \bar{v}_y q^{(n, N, \theta)} + \bar{v} q^{(n, N, \theta)}_y) \\ \label{weak:3}
& + \kappa (U(u^0), q^{(n, N, \theta)}) - (\bar{u}_\theta^{(N)} \p_x \bar{u}_\theta^{(N)} q^{(n, N, \theta)}_y, q^{(n ,N, \theta)}_y) = (F^{(\theta)}, q^{(n, N, \theta)}). 
\end{align}

We now introduce some notation to clarify the next two lemmas: 
\begin{align*}
&\mathcal{E}_{(k)}(x) := \| \bar{u}_\theta^{(N)} \p_x^{k-1} q^{(n, \theta, N)}_{yy} \|_{L^2_y} + \| \bar{u}_\theta^{(N)} \p_x^k q^{(n, \theta, N)}_{y} \|_{L^2_y} \\
&\mathcal{D}_{(k)}(x) :=  \| \sqrt{\bar{u}} \p_x^{k-1} q^{(n, \theta, N)}_{yyy} \|_{L^2_y} + \| \sqrt{\bar{u}} \p_x^k q^{(n, \theta, N)}_{yy} \|_{L^2_y}
\end{align*}

\begin{lemma} The function $q^{(n, N, \theta)}$ satisfies the following inequality for $k \ge 0$, 
\begin{align} \n 
&\sup_{x} \| \bar{u}_\theta^{(N)} \p_x^k q^{(n, N, \theta)}_y \|_{L^2_y}^2 + \| \sqrt{\bar{u}} \p_x^k q^{(n, N, \theta)}_{yy} \|_{L^2_y}^2 + \bar{u}_y(0) |\p_x^k q^{(n, N, \theta)}_y(0)|^2 \\ \label{energy:gal:3}
&\le C_\theta \|  \p_x^k F^{(\theta)} \langle y \rangle \|^2  + \| u^0 \|_{H^1_y}^2 + \| \bar{u}_\theta^{(N)} \p_x^k q^{(n, N, \theta)}_y|_{x = 0} \|_{L^2_y}^2.
\end{align}
\end{lemma}
\begin{proof} Our starting point is \eqref{weak:3}. By expanding the second term above using $v^{(n, \theta, N)} = \bar{u} q^{(n, \theta, N)}$, we obtain 
\begin{align} \n
&(\p_{yy}(\bar{u} q^{(n, N, \theta)}), q^{(n, N, \theta)}_{yy}) = \| \sqrt{\bar{u}} q^{(n, N, \theta)}_{yy} \|_{L^2_y}^2 - \bar{u}_y(0) |q^{(n, N, \theta)}_y(0)|^2 \\ \label{energy:gal:k}
& - 2(\bar{u}_{yy} q^{(n, N, \theta)}_y, q^{(n, N, \theta)}_y) + \frac 1 4 (\bar{u}_{yyyy} q^{(n, N, \theta)}, q^{(n, N, \theta)}).
\end{align}
The contribution at $y = 0$ is absorbed into the third term from \eqref{weak:3} upon using the prefactor of $2$. 

We now estimate 
\begin{align} \n
&|2(\bar{u}_{yy} q^{(n, N, \theta)}_y, q^{(n, N, \theta)}_y) - \frac 1 4 (\bar{u}_{yyyy} q^{(n, N, \theta)}, q^{(n, N, \theta)})| \le C_\theta \| \bar{u}_\theta^{(0)} q^{(n, N, \theta)}_y \|_{L^2_y}^2. 
\end{align}
We next, estimate the fourth through eighth terms from \eqref{weak:3} above by \\
$\kappa C_\theta \mathcal{E}_0(x)^{\frac 1 2} (\mathcal{D}_0(x)^{\frac 1 2} + \| u^0 \|_{H^1_y})$, and the ninth term from \eqref{weak:3} by $o_L(1) C_\theta \sup_x \mathcal{E}_0(x)$, upon invoking that $\theta > 0$. 

From here, we achieve our basic energy identity 
\begin{align} \n
&\sup_{x} \| \bar{u}_\theta^{(N)} q^{(n, N, \theta)}_y \|_{L^2_y}^2 + \| \sqrt{\bar{u}} q^{(n, N, \theta)}_{yy} \|_{L^2_y}^2 + \bar{u}_y(0) |q^{(n, N, \theta)}_y(0)|^2 \\ \n
&\le C_\theta \| \hat{F} \langle y \rangle \|^2 + o_L(1) \| u^0 \|_{H^1_y}^2 + o_L(1) \sup_x \| \bar{u}_\theta^{(0)} q^{(n, N, \theta)}_y \|_{L^2_y}^2 + \| \bar{u}_\theta^{(0)} q^{(n, N, \theta)}_y|_{x = 0} \|_{L^2_y}^2,
\end{align}
which, for $L$ sufficiently small, ensures the estimate \eqref{energy:gal:3}. This concludes the proof of estimate \eqref{energy:gal:3}

Estimate \eqref{energy:gal:k} follows in essentially an identical manner to \eqref{energy:gal:3} upon differentiating $k$ times in $x$. The only exception is we need to treat 
\begin{align} \n
(\p_x^k v^{(n, N, \theta)}_{yy}, \p_x^k q^{(n, N, \theta)}_{yy}) = & \| \sqrt{\bar{u}} q^{(n, N, \theta)}_{yy} \|_{L^2_y}^2 + 2( \bar{u}_y \p_x^k q^{(n, N, \theta)}_y, \p_x^k q^{(n, N, \theta)}_{yy}) \\ \n
&+ (\bar{u}_{yy} \p_x^k q^{(n, \theta, N)}, \p_x^k q^{(n, \theta, N)}_{yy}) + \mathcal{C}_2, 
\end{align}
where the commutators in $\mathcal{C}_2$ are defined by 
\begin{align} \n
\mathcal{C}_2 := & \sum_{j = 1}^k \binom{k}{j} \Big( (\p_x^j \bar{u}_{yy} \p_x^{k-j} q^{(n, \theta, N)}, \p_x^k q^{(n, \theta, N)}_{yy}) + 2( \p_x^j \bar{u}_y \p_x^{k-j} q^{(n, \theta, N)}_y ,\p_x^k q^{(n, \theta, N)}_{yy} ) \\
 &\qquad \qquad + ( \p_x^j \bar{u} \p_x^{k-j} q^{(n, \theta, N)}_{yy},\p_x^k q^{(n, \theta, N)}_{yy}) \Big).
\end{align}
We can easily estimate now
\begin{align}
|\mathcal{C}_2| \le C_\theta \mathcal{E}_{(k)}(x)^{\frac 1 2} \mathcal{D}_{(k)}(x)^{\frac 1 2},
\end{align}
which concludes the proof of the lemma. 
\end{proof}

\begin{lemma} The function $q^{(n, N, \theta)}$ satisfies the following inequality for $k \ge 1$: 
\begin{align} \n
&\sup_x  \| \bar{u}_\theta^{(N)} \p_x^{k-1} q^{(n, \theta, N)}_{yy} \|_{L^2_y}^2 +  \| \sqrt{\bar{u}} \p_x^{k-1} q^{(n, \theta, N)}_{yyy} \|^2\\ \label{energy:gal:2}
 \le & C_\theta \|  \p_x^k F^{(\theta)} \langle y \rangle \|^2  + \| u^0 \|_{H^1_y}^2 + \| \bar{u}_\theta^{(N)} \p_x^k q^{(n, N, \theta)}_y|_{x = 0} \|_{L^2_y}^2. 
\end{align}
\end{lemma}
\begin{proof} We start from the weak formulation, \eqref{weak:2}, which, together with $e_i'' = (\frac{2\pi i}{N})^2 e_i$, immediately implies the validity of the following identity 
\begin{align} \n
&-\frac{\p_x}{2}( |\bar{u}_\theta^{(N)}|^2 q^{(n, N, \theta)}_y, q^{(n, N, \theta)}_{yyy}) - (\p_{yy}(\bar{u} q^{(n, N, \theta)}), q^{(n, N, \theta)}_{yyyy}) + 2 \bar{u}_y |q^{(n, N, \theta)}_y(0)|^2 \\ \n
&- \kappa (\bar{v}_{xyy} I_x[v^{(n, N, \theta)}_y], q_{yy}^{(n, N, \theta)}) - \kappa (\bar{v}_{yy} v^{(n, N, \theta)}_y, q_{yy}^{(n, N, \theta)}) \\ \n
& - \kappa (I_x[v^{(n, N, \theta)}_{yy}], \bar{v}_{xy}q_{yy}^{(n, N, \theta)} + \bar{v}_x q^{(n, N, \theta)}_{yyy})  - \kappa (v^{(n, N, \theta)}_{yy}, \bar{v}_y q_{yy}^{(n, N, \theta)} - \bar{v} q^{(n, N, \theta)}_{yyy}) \\ \label{weak:10}
&  - \kappa (U(u^0), q_{yy}^{(n, N, \theta)}) + (\bar{u}_\theta^{(N)} \p_x \bar{u}_\theta^{(N)} q^{(n, N, \theta)}_y, q^{(n ,N, \theta)}_{yy}) = - (F^{(\theta)}, q_{yy}^{(n, N, \theta)}). 
\end{align}
From here, we can integrate by parts both terms using also that for each fixed $n < \infty$, $q_{yy}|_{y = 0} = q_{yy}|_{y = N} = 0$, which produces for the first two terms above 
\begin{align} \n
\sum_{j = 1}^2 (\ref{weak:10}.j) = &\frac{\p_x}{2} \| \bar{u}_\theta^{(N)} q^{(n, \theta, N)}_{yy} \|_{L^2_y}^2 + (\p_y \bar{u}_\theta^{(N)} q^{(n, \theta, N)}_{xy}, q^{(n, \theta, N)}_{yy}) + (v^{(n, \theta, N)}_{yyy}, q^{(n, \theta, N)}_{yyy}).
\end{align}   
Expanding out $v^{(n, \theta, N)}_{yyy} = \p_y^3(\bar{u} q^{(n, \theta, N)})$, we obtain furthermore the following identity 
\begin{align} \n
\sum_{j = 1}^2 (\ref{weak:10}.j) = &\frac{\p_x}{2} \| \bar{u}_\theta^{(N)} q^{(n, \theta, N)}_{yy} \|_{L^2_y}^2 + \| \sqrt{\bar{u}} q^{(n, N, \theta)}_{yyy} \|_{L^2_y}^2 + \mathcal{R}_1, 
\end{align}
where we have defined 
\begin{align} \n
\mathcal{R}_1 :=& - \frac 9 2 (\bar{u}_{yy} q^{(n, \theta, N)}_{yy}, q^{(n, \theta, N)}_{yy}) - 4 (\bar{u}_{yyy} q^{(n, \theta, N)}_y, q^{(n, \theta, N)}_{yy} ) \\
&- (\bar{u}_{yyyy} q^{(n, \theta, N)}, q^{(n, \theta, N)}_{yy}) + (\p_y \bar{u}_\theta^{(N)} q^{(n, \theta, N)}_{xy}, q^{(n, \theta, N)}_{yy}). 
\end{align}
We can see that we may easily estimate 
\begin{align}
|\mathcal{R}_1| \lesssim \| q_{yy}^{(n, \theta, N)} \|_{L^2_y}( \| q_{yy}^{(n, \theta, N)} \|_{L^2_y} + \| q_{xy}^{(n, \theta, N)} \|_{L^2_y} ). 
\end{align}
We next estimate in a straightforward manner using that $\theta > 0$, 
\begin{align} \n
|\sum_{j = 4}^{9} (\ref{weak:10}.j)| \le C_\theta \mathcal{E}_1(x)^{\frac 1 2} (\mathcal{D}_1(x)^{\frac 1 2} + \| u^0 \|_{H^2_y}).  
\end{align}
This concludes the proof of the lemma for $k = 1$. The $k > 1$ case works in a largely identical manner. 
\end{proof}

We note that our compatibility conditions from Definition \ref{def:compatibility:condition} imply that higher order $x$ derivatives of the equation \eqref{theta:approx:1} holds at $\{x = 0\}$, and therefore we can evaluate the ODE \eqref{ODE:x:time} to obtain the initial condition for $\dot{b}^n_i$ (and higher order derivatives thereof). Combining \eqref{energy:gal:3}, \eqref{energy:gal:2}, and \eqref{energy:gal:3}, we may thus send $n \rightarrow \infty$ and $N \rightarrow \infty$ to obtain a distributional solution to \eqref{theta:approx:1}, which can then be upgraded to a strong solution in the usual manner. We state this now as a lemma:
\begin{lemma} Fix any $\theta > 0$. Assume compatibility conditions on $q^{(\theta)}|_{x= 0}$ according to Definition \eqref{def:compatibility:condition}. Assume $F^{(\theta)}$ satisfies $\sum_{j = 0}^{M_0}\| \p_x^j F^{(\theta)} \langle y \rangle \| < \infty$. Then there exists an $L = L(\theta) > 0$ such that there exists a unique solution to \eqref{theta:approx:1} on the interval $(0, L)$. 
\end{lemma}

\subsection{Uniform Estimates in $\theta$}

In this section, we provide uniform in $\theta$ estimates for the problem \eqref{theta:approx:1}. As such, we shall define the $\theta$-dependent version of our basic $X$ norm
\begin{align} \n
\| q \|_{X_\theta} := &   \sup_{0 \le x_0 \le L} \Big( \|\bar{u}_\theta q_{xy} \|_{x = x_0} + \|q_{yyy} w (1 - \chi) \|_{x= x_0} \Big)  \\
&+ \| \sqrt{\bar{u}} q_{xyy} w \| + \| v_{yyyy} w \|, \\
\| q \|_{X_{k, \theta}} := & \sum_{j = 0}^k \| q^{(j)} \|_{X_{\theta}}. 
\end{align} 
We adopt the notation that 
\begin{align} \label{def:pk}
p_{k} := p(\| \bar{q} \|_{X_k}), \qquad p_{\langle k \rangle} := p(\| \bar{q} \|_{X_{\langle k \rangle}}),
\end{align}
where $p$ is an inhomogeneous polynomial of one variable of unspecified power. In general, we will suppress those constants which depend on $\| \bar{q} \|_{X_{\langle k \rangle}}$, and only display those which depend on $\| \bar{q} \|_{X_{\langle k + 1 \rangle}}$.

\begin{lemma} \label{lemma:Sobolev}  The following inequalities are valid: 
\begin{subequations}
\begin{align} \label{int:1}
&\| q^{(k)}_{xy} \| + \| \{q^{(k)}_y, q^{(k)}_{yy} \} w \| + \| q^{(k)} \langle y \rangle^{-1} \| \le o_L(1)(1 + \| q \|_{X_{k, \theta}}) \\  \label{int:2}
& \|v^{(k)} \langle y \rangle^{-1} \| + \| \{v^{(k)}_{y}, v^{(k)}_{yy}, v^{(k)}_{yyy}, v^{(k)}_{xy} \} w \| \le o_L(1) (1 + \| q \|_{X_{k, \theta}}) \\  \label{int:3}
&\| q^{(k)}_{xy} w \| + \| v^{(k)}_{xyy} w \| \lesssim 1 + \| q \|_{X_{k, \theta}} \\  \label{int:4}
&\| q^{(k)} \|_\infty + \| v^{(k)}_y \|_\infty + \| q^{(k)}_y \|_{\infty} \lesssim 1 + o_L(1) \| q \|_{X_{\langle k, \theta \rangle}} \\  \label{int:5}
&\| v^{(k)}_{yy}, v^{(k)}_{yyy} \|_\infty \le 1 + o_L(1) \| q \|_{X_{\langle k + 1, \theta \rangle}}.
\end{align}
\end{subequations}
\end{lemma}
\begin{proof} The first step is to obtain control over $\| q^{(k)}_{xy} \|$ via interpolation. 
\begin{align*}
\| q^{(k)}_{xy} \{1 - \chi(\frac{y}{\delta}) \} \| \le \delta^{-1} \| \bar{u}_\theta q^{(k)}_{xy} \| \le L \delta^{-1} \sup_x \|\bar{u}_\theta q^{(k)}_{xy}\|_{L^2_y} \le L \delta^{-1} \| q \|_{X_{k, \theta}}. 
\end{align*}

\noindent Near the $\{y = 0\}$ boundary, one interpolates: 
\begin{align*}
|( \chi(\frac{y}{\delta}) \p_y \{ y \}, |q^{(k)}_{xy}|^2)| \lesssim &\| \chi y q^{(k)}_{xyy} \|^2 + ( \frac{y}{\delta} \chi'(\frac{y}{\delta}), |q^{(k)}_{xy}|^2) \\
\lesssim & \delta \| q^{(k)} \|_{X_\theta}^2 + L^2 \delta^{-2} \| q^{(k)} \|_{X_\theta}^2. 
\end{align*}

\noindent  Optimizing $\sqrt{\delta} + L \delta^{-1}$, one obtains $\delta = L^{2/3}$. Thus, $\| q^{(k)}_{xy} \| \lesssim L^{1/3} \| q \|_{X_{k, \theta}}$. From here, a basic Poincare inequality gives: 
\begin{align*}
\| q^{(k)}_y \| = \| q^{(k)}_y|_{x = 0} + \int_0^x q^{(k)}_{xy} \| \lesssim \sqrt{L} |q^{(k)}_y|_{x = 0}\| + L \| q^{(k)}_{xy} \|
\end{align*}

\noindent  From here, Hardy inequality gives immediately $\| q^{(k)} \langle y \rangle^{-1} \| \le \| q^{(k)}_y \|$. 

The next step is to establish the uniform bound via straightforward Sobolev embedding: 
\begin{align*}
|q^{(k)}|^2 \lesssim \sup_x |q^{(k)}_y \langle y \rangle\|^2 \lesssim |q^{(k)}_y|_{x = 0} \langle y \rangle\|^2 + L \| q^{(k)}_{xy} \langle y \rangle \|^2 \lesssim 1 + o_L(1) \| q \|_{X_{k, \theta}}. 
\end{align*}

A Hardy computation gives:  
\begin{align*}
\| q^{(k)}_{xy} w (1 - \chi) \| \lesssim& \| q^{(k)}_{xy} \|_{2,loc} + \| q^{(k)}_{xyy} w (1 - \chi )\| \lesssim \| q \|_{X_{k, \theta}}.
\end{align*}

We record the following expansions which follow from the product rule upon recalling that $v = \bar{u} q$:
\begin{align}
\begin{aligned} \label{expressionsk}
&|v^{(k)}| \lesssim \sum_{j = 0}^k |\bar{u}^j q^{(k-j)}| \\
&|v^{(k)}_y| \lesssim \sum_{j = 0}^k |\bar{u}^j_{y} q^{(k-j)}| + |\bar{u}^j q^{(k-j)}_y| \\
&|v^{(k)}_{yy}| \lesssim \sum_{j = 0}^k |\bar{u}^j_{yy} q^{(k-j)}| + |\bar{u}^j_{y} q^{(k-j)}_{y}| + |\bar{u}^j q^{(k-j)}_{yy}| \\
&|v^{(k)}_{yyy}| \lesssim \sum_{j = 0}^k |\bar{u}^j_{yyy} q^{(k-j)}| + |\bar{u}^j_{yy} q^{(k-j)}_y| + |\bar{u}^j_{y} q^{(k-j)}_{yy}|+ |\bar{u}^j q^{(k-j)}_{yyy}| \\
&|v^{(k)}_{xy}| \lesssim \sum_{j = 0}^k |\bar{v}^{j}_{yy} q^{(k-j)}| + |\bar{u}^j_{y} q^{(k-j)}_x| + |\bar{v}^j_{y} q^{(k-j)}_y| + |\bar{u}^j q^{(k-j)}_{xy}| \\
&|v^{(k)}_{xyy}| \lesssim \sum_{j = 0}^k |\bar{v}^j_{yyy} q^{(k-j)}| + |\bar{u}^j_{yy} q^{(k-j)}_x| + |\bar{v}^j_{yy} q^{(k-j)}_y| + |\bar{u}^j_{y} q^{(k-j)}_{xy}| \\
& \hspace{10 mm} + |\bar{v}^j_{y} q^{(k-j)}_{yy}| + |\bar{u}^j q^{(k-j)}_{xyy}|.
\end{aligned}
\end{align}

We will restrict to $k = 0$ for the remainder of the proof, as the argument works for general $k$ in a straightforward way. From (\ref{expressionsk}), $\| v_y\|$ follows obviously. Next, 
\begin{align} \label{broad}
\| v_{yy} \| \lesssim & \| \bar{u}_{yy} q \| + \| \bar{u}_{y} q_y \| + \| \bar{u} q_{yy} \| \lesssim  \sqrt{L} + o_L(1) \| q \|_{X_\theta}. 
\end{align}

\noindent  From here, $\| v_{yyy} \|_{loc}$ can be interpolated in the following way: 
\begin{align*}
&(v_{yyy}, v_{yyy} \chi(\frac{y}{\delta})) = ( \p_y \{ y \} \chi(\frac{y}{\delta}), |v_{yyy}|^2) \\
= & - ( y \chi(\frac{y}{\delta}) v_{yyy}, \p_x^{k-1} v_{yyyy}) - ( y \delta^{-1} \chi'(\frac{y}{\delta}), |v_{yyy}|^2) \\
\lesssim &\delta^2 \| v_{yyyy} \|^2 + \| \psi_\delta v_{yyy} \|^2. 
\end{align*}

\noindent  For the far-field component, we may majorize via: 
\begin{align*}
|(v_{yyy}, v_{yyy} \{1 - \chi(\frac{y}{\delta}) \} )| \lesssim \| \psi_\delta v_{yyy} \|^2.
\end{align*}

\noindent  Here $\psi_\delta = 1 - \chi(\frac{10 y}{\delta})$, the key point being that both $\{1 - \chi(\frac{y}{\delta})\}$ and $\chi'(\frac{y}{\delta})$ are supported in the region where $\psi_\delta = 1$. To estimate this term, we may integrate by parts: 
\begin{align*}
( \psi v_{yyy}, v_{yyy}) = &- ( \psi v_{yy}, v_{yyyy})  - ( \delta^{-1} \psi' v_{yy}, v_{yyy}) \\
= & - ( \psi v_{yy}, v_{yyyy}) + ( \delta^{-2} \frac{\psi''}{2}, |v_{yy}|^2) \\
\lesssim & \delta^2 \|  v_{yyyy} \|^2 + N_\delta \| v_{yy} \|^2. 
\end{align*}

\noindent  Thus, 
\begin{align} \label{interp.3}
\| v_{yyy} \| \le & \delta \| v_{yyyy} \| + N_\delta \| v_{yy} \|.
\end{align}

\noindent  We combine the above with (\ref{broad}) to select $\delta = L^{0+}$ to achieve control over $\|\p_y^j v^{(k)}\|$ for $j = 1,2,3$. 

We now use the elementary formula $\frac{1}{a+b} = \frac{1}{a} - \frac{b}{a+b}$ to write: 
\begin{align*}
q = \frac{v}{u_s} = \frac{v}{u_{sy}(0) y + [u_s - u_{sy}(0)y]} = \frac{1}{u_{sy}(0)} \frac{v}{y} - v \frac{u_s - u_{sy}(0)y}{y u_{sy}(0) u_s}.
\end{align*}

Using the estimates $u_s \gtrsim y$ as $y \downarrow 0$ and $|u_s - u_{sy}(0)y| \lesssim y^2$ as $y \downarrow 0$, it is easy to see that the second quotient above is bounded and in fact $\mathcal{C}^\infty$. We may thus limit our study to $q_0 := \frac{v}{y}$. We let $k_1 + k_2 = 3$ and differentiate the formula:
\begin{align} \n
q_0(x,y) = \frac{1}{y} \int_0^y v_y(x,y') \ud y'  = \int_0^1 v_y(x,ty) \ud t, 
\end{align}
where we changed variables via $ty = y'$. From here, we can obtain: 
\begin{align}
\| q_{yy} \chi(y) \| \lesssim \| \int_0^1 t^2 v_{yyy}(x, ty) \ud t \|_{loc} \lesssim  \| v_{yyy} \|. 
\end{align}

\noindent  Away from the $\{y = 0\}$ boundary, we estimate trivially: 
\begin{align*}
\| q_{yy}\{1 - \chi(\frac{y}{\delta}) \} w \| &\lesssim \| \bar{u} q_{yy} w \{1 - \chi(\frac{y}{\delta}) \} \|_{2} \\
& \lesssim \sqrt{L} \|\bar{u} q_{yy} w\|_{x = 0} + L \| \sqrt{\bar{u}} q_{xyy} w \|. 
\end{align*}

\noindent  From here, obtaining $\| q_y w \|$ follows from Hardy. We now turn our attention to the weighted estimates for $v_y, v_{yy}, v_{xy}, v_{xyy}$, which follow from (\ref{expressionsk}), whereas for $v_{yyy}$, we use the Prandtl equation to produce the identity: 
\begin{align*}
v_{yyy} = & \bar{u}_{yyy} q + 3 \bar{u}_{yy} q_y + 3 \bar{u}_{y} q_{yy} + \bar{u} q_{yyy} \\
= &  \Big( - \bar{u} \bar{v}_{yy} + \bar{v} \bar{u}_{yy} \Big) q +  3 \bar{u}_{yy} q_y + 3 \bar{u}_{y} q_{yy} + \bar{u} q_{yyy}
\end{align*}

The uniform estimates subsequently follow from straightforward Sobolev embeddings.
\end{proof}

We now perform our first energy estimate. 
\begin{lemma}[$\mathcal{E}_k$ Estimate] \label{energy:estImate:1089}  Assume $q$ solves \eqref{origPrLay.beta}. Then the following inequality is valid: 
\begin{align}
\begin{aligned} \label{Ek}
\| q \|_{\mathcal{E}_{k, \theta}}^2 \lesssim & \|\bar{u} \p_x^k q_{xy}|_{x = 0}\|^2  + o_L(1)(p_{\langle k \rangle} + \kappa p_{\langle k + 1 \rangle}) (1 + \| q \|_{X_{\langle k, \theta \rangle}}^2 ) \\
& + o_L(1) C(u^0) + o_L(1) \| \p_x^{k+1} \p_{xy}g_1 \langle y \rangle \|^2.
\end{aligned}
\end{align}
\end{lemma}
\begin{proof} We have the following energy inequality,
\begin{align} \label{est.k.weak.int}
\frac{\p_x}{2} \| \bar{u}_\theta \p_x^{k+1} q^{(\theta)}_y \|_x^2 + \| \sqrt{\bar{u}_\theta} \p_x^{k+1} q_{yy}^{(\theta)} \|^2 + 2 \bar{u}_{y}(0) \Big(\p_x^{k+1} q^{(\theta)}_y(0) \Big)^2 \le |\mathcal{I}_{k+1}^{(\theta)}|,
\end{align}

\noindent where the integral $\mathcal{I}^{(\theta)}_{k+1}$ is defined 
\begin{align} \n
\mathcal{I}_{k+1}^{(\theta)} :=& \sum_{j = 1}^{k+1} \binom{k+2}{j} \Big( (\p_x^j \{ \bar{u}_\theta^2 \} \p_x^{k+2-j} q^{(\theta)}_y, \p_x^{k+1} q^{(\theta)}_y ) + (\p_x^j \bar{u} \p_x^{k+1-j} q^{(\theta)}_{yy}, \p_x^{k+1} q^{(\theta)}_{yy}) \Big) \\ \n
& + \sum_{j = 0}^{k+1} \binom{k+1}{j} (2 \p_x^j \p_y \bar{u} \p_x^{k+1-j} q^{(\theta)}_{y}, \p_x^{k+1} q^{(\theta)}_{yy}) \\ \n
& + \sum_{j = 0}^{k+1} \binom{k+1}{j} (\p_x^j \p_y^2 \bar{u} \p_x^{k+1-j} q^{(\theta)}, \p_x^{k+1} q^{(\theta)}_{yy}) + \kappa( \p_x^{k+1} U, \p_x^{k+1} q^{(\theta)})  \\ \n
& +\kappa  \sum_{j = 0}^{k+1} \binom{k+1}{j} (\p_x^j \bar{v}_{xyy} I_x[\p_x^{k+1-j} v^{(\theta)}_y] + \p_x^j \bar{v}_{yy} \p_x^{k+1-j} v^{(\theta)}_y, \p_x^{k+1} q^{(\theta)}) \\ \n
& + \kappa \sum_{j = 0}^{k+1} \binom{k+1}{j} \Big( (I_x[   \p_x^{k+1-j} v^{(\theta)}_{yy} ], \p_y \{ \p_x^j \bar{v}_x \p_x^{k+1} q^{(\theta)} \}) \\ \n
& + (\p_x^{k+1-j} v^{(\theta)}_{yy}, \p_y \{ \p_x^j \bar{v} \p_x^{k+1} q^{(\theta)} \}) \Big)  + 2\sum_{j =1}^{k+1} \binom{k+1}{j} \p_x^j \bar{u}_{y} \p_x^{k+1-j} q^{(\theta)}_y(0) \p_x^{k+1} q^{(\theta)}_y(0) \\ \label{Inthetak}
&+ (\p_x^{k+1} F^{(\theta)}, \p_x^{k+1} q^{(\theta)}) =  \sum_{\ell = 1}^{11} \mathcal{I}^{(\theta)}_{k+1,\ell}.
\end{align}

For this estimate, we further integrate over $x \in [0, x_0]$, where $0 < x_0 < L$. This produces 
\begin{align} \label{est.k.weak.int:energy}
\sup_x \| \bar{u} \p_x^{k+1} q_y \|_x^2 + \| \sqrt{\bar{u}} \p_x^{k+1} q_{yy} \|^2 + 2 \| \sqrt{\bar{u}_{y}} \p_x^{k+1} q_y \|_{y = 0}^2 \le \int_0^L |\mathcal{I}_{k+1}^{(\theta)}|,
\end{align}

We recall now the various contributions into $\mathcal{I}_{k+1}^{(\theta)}$ from \eqref{Inthetak}. We start with the first of these, via  
\begin{align}\n
\int_0^{L} |\mathcal{I}_{k+1, 1}^{(\theta)}|  = &\int_0^L \Big| \sum_{j = 1}^{k+2} \binom{k+2}{j} \sum_{j_1 = 0}^j \binom{j}{j_1} (\p_x^{j_1} \bar{u}_\theta \p_x^{j - j_1} \bar{u}_\theta \p_x^{k+2-j} q_y, \p_x^{k+1} q_y) \Big| \\ \n
 \lesssim & \Big\| \frac{\p_x^{\langle \frac{k+1}{2} \rangle} \bar{u}_\theta}{\bar{u}_\theta} \Big\|_\infty^2 \| \bar{u}_\theta \p_x^{\langle k + 1 \rangle} q_y \|^2 \\ \n
& + \Big\| \frac{\p_x^{\langle \frac{k+2}{2} \rangle}\bar{u}_\theta}{\bar{u}_\theta} \Big\|_\infty \| \p_x^{\langle \frac{k+2}{2} \rangle} q_y \|_{L^\infty_x L^2_y} \| \p_x^{k+2} \bar{u}_\theta \|_{L^2_x L^\infty_y} \| \bar{u}_\theta \p_x^{k+1} q_y \| \\ \n
\lesssim & o_L(1) p(\| \bar{q} \|_{X_{\langle k \rangle}}) \| q \|_{X_{\langle k \rangle}}^2.
\end{align}
We have used the fact that $\frac{\p_x^{j} \bar{u}_\theta}{\bar{u}_\theta} \sim \p_x^{j-1} \p_{yy} \bar{v}$ for any $j$, in particular with $j = \langle \frac{k+2}{2} \rangle$ and then appealed to \eqref{int:4}. We also use \eqref{int:1} to bound $\| \p_x^j q_y\|$ for $j = \frac{k+2}{2}$.

We now move to $\p_y^4$ term from the specification of $\mathcal{I}^{(\theta)}_{k+1}$, the first of which reads 
\begin{align} \label{sub1.1}
\int_0^L | \mathcal{I}_{k+1,2}^{(\theta)} | \le\int_0^L | \sum_{j = 1}^{k+1} \binom{k+1}{j} (\p_x^j \bar{u} \p_x^{k+1-j} q_{yy}, \p_x^k q_{yy}) |.
\end{align}

\noindent The $j = k+1$ case from above contributes 
\begin{align*}
|\int_0^{x_0} (\p_x^{k+1} \bar{u} q_{yy}, \p_x^{k+1} q_{yy})| \le & \Big\| \frac{\p_x^{k+1} \bar{u}}{\bar{u}} \Big\|_{L^2_x L^\infty_y} \| \sqrt{\bar{u}} q_{yy} \|_{L^\infty_x L^2_y} \| \sqrt{\bar{u}} \p_x^{k+1} q_{yy} \| \\
\lesssim & ( 1 + o_L(1) \| \bar{q} \|_{X_{\langle k \rangle}})(1 + \| q \|_{X_{\langle k \rangle}}) \| \sqrt{\bar{u}} \p_x^{k} q_{xyy} \|. 
\end{align*}

\noindent Above, we have used the Hardy and Agmon inequalities, as well as \eqref{int:2}, which gives
\begin{align*}
\Big\| \frac{\p_x^{k+1} \bar{u}}{\bar{u}} \Big\|_{L^2_x L^\infty_y} = & \Big\| \frac{\p_x^{k} \bar{v}_y}{\bar{u}} \Big\|_{L^2_x L^\infty_y} \lesssim \| v^{(k)}_{yy} \|_{L^2_x L^\infty_y} +  \| v^{(k)}_y \|_{L^2_x L^\infty_y} \\
\lesssim & \| v^{(k)}_{yyy} \| + \| v^{(k)}_{yy} \| +  \| v^{(k)}_y \| \lesssim o_L(1)  (1 + \| \bar{q} \|_{X_{\langle k \rangle}}).
\end{align*}

\noindent The intermediate cases, $j = 1,...,k$ can be bounded above by
\begin{align*}
&\Big\| \frac{\p_x^{\langle k \rangle} \bar{u}}{\bar{u}} \Big\|_{\infty} \| \sqrt{\bar{u}} \p_x^{\langle k \rangle} q_{yy} \|  \| \sqrt{\bar{u}} \p_x^k q_{xyy} \| \lesssim  (1 + o_L(1) p_{k}) \| q \|_{X_k} \| q \|_{X_{\langle k \rangle}}. 
\end{align*} 

The next term is 
\begin{align} \label{sub1.2}
\int_0^{L} |\mathcal{I}_{k+1, 3}^{(\theta)}| = &\int_0^{L} |\sum_{j = 0}^{k+1} \binom{k+1}{j} (2 \p_x^j \p_y \bar{u} \p_x^{k+1-j} q_y, \p_x^{k+1} q_{yy}) |\\ \n
= &\int_0^{L} | ( 2 \bar{u}_{y} \p_x^{k+1} q_y, \p_x^{k+1} q_{yy} ) | + \sum_{1 \le j \le \frac{k+1}{2}} \int_0^{L} |( \p_x^j \bar{u}_{y} \p_x^{k+1 - j} q_y, \p_x^{k+1} q_{yy} )| \\ \n
& + \sum_{\frac{k+1}{2} \le j \le k+1} \int_0^{L} |( \p_x^j \bar{u}_{y} \p_x^{k+1-j} q_y, \p_x^{k+1} q_{yy} ) |. 
\end{align}

First, using $\bar{u}_{yy} = \bar{U}_0''(y) - I_x[\bar{v}_{yyy}]$ and also integration by parts gives: 
\begin{align*}
(\ref{sub1.2}.1) \le &  \int_0^{L} |( \bar{u}_{yy} \p_x^{k+1} q_y, \p_x^{k+1} q_y )| + \int_0^{L} \bar{u}_{y}(0) |\p_x^{k+1} q_y(0)|^2 \\
\le & o_L(1)(1 + p_{\langle2 \rangle}) \| q \|_{X_k}^2 + \|\sqrt{\bar{u}_y} |\p_x^{k+1}q_y \|_{y = 0}^2.  
\end{align*}

\noindent We now absorb that the crucial boundary term above can be absorbed into the left-hand side of \eqref{est.k.weak.int:energy} due to the factor of $2$ in \eqref{est.k.weak.int:energy} as compared with the factor of $1$ above. 

 We now treat, where we have used again that $\bar{u}_x = - \bar{v}_y$ and subsequently \eqref{int:1} - \eqref{int:2} to bound the $\bar{u}$ contributions: 
\begin{align*}
(\ref{sub1.2}.2) \le &\int_0^{L} |( \p_x^{k+1} q_y, \p_x^j \bar{u}_{yy} \p_x^{k+1-j} q_y )| + \int_0^{L} |( \p_x^{k+1} q_y, \p_x^j \bar{u}_{y} \p_x^{k+1 -j} q_{yy} )| \\
& + \int_0^{L}| \p_x^j \bar{u}_{y} \p_x^{k-j+1} q_y \p_x^{k+1} q_y(0)| \\
\le  & \int_0^{L} | ( \p_x^{k+1} q_y, \p_x^j \bar{u}_{yy} \p_x^{k+1-j} q_y )| + \int_0^{L}| (  \p_x^{k+1} q_y, \p_x^j \bar{u}_{y} \p_x^{k+1 -j} q_{yy} )|\\
& + \int_0^{L}| \p_x^j \bar{u}_{y} \p_x^{k-j+1} q_y \p_x^{k+1} q_y(0)| \\
\lesssim & \| \p_x^{\langle \frac{k+1}{2}\rangle} \bar{u}_{yy} \|_\infty \| \p_x^{\langle k \rangle} q_y \| \| \p_x^{k+1} q_y \| + \| \p_x^{\langle \frac{k+1}{2} \rangle} \bar{u}_{y} \|_\infty \| \p_x^{\langle k \rangle} q_{yy} \| \| \p_x^{k+1} q_{y} \| \\
& + o_L(1) \| q \|_{X_{\langle k \rangle}}^2 \\
\lesssim & o_L(1) p_{\langle k \rangle} ( \| q \|_{X_{\langle k \rangle}}^2 + 1 ),
\end{align*}

Similarly, we obtain by invoking the definition of $\| \bar{q} \|_{X_k}$, as well as \eqref{int:1}
\begin{align*}
(\ref{sub1.2}.3) \le &\int_0^{L} |(\p_x^{k+1} \bar{u}_{yy} q_y, \p_x^{k+1} q_y)| + \int_0^{L}| (\p_x^{k+1} \bar{u}_y q_{yy}, \p_x^{k+1}q_y)|\\
&  + \int_0^{L}| \sum_{j = \frac{k+1}{2}}^{k} (\p_x^j \bar{u}_{yy} \p_x^{k+1-j} q_y, \p_x^{k+1}q_y) | + \int_0^{L}| \sum_{j = \frac{k+1}{2}}^{k}(\p_x^j \bar{u}_y \p_x^{k+1-j} q_{yy}, \p_x^{k+1}q_y)| \\
\lesssim & \| \p_x^k \bar{v}_{yyy} \|_{L^2_x L^\infty_y} \| q_y \|_{L^\infty_x L^2_y} \| \p_x^{k+1}q_y \| + \| \p_x^k \bar{v}_{yy} \|_{L^2_x L^\infty_y} \| q_{yy} \|_{L^\infty_x L^2_y} \| \p_x^{k+1} q_{yy} \| \\
& + \| \p_x^{\langle k - 1 \rangle} \bar{v}_{yyy} \|_{L^\infty} \| \p_x^{\langle \frac{k+1}{2} \rangle} q_y \| \| \p_x^{k+1} q_y \|  + \| \p_x^{\langle k - 1 \rangle} \bar{v}_{yy} \|_{L^\infty} \| \p_x^{\langle  \frac{k+1}{2} \rangle} q_{yy} \| \| \p_x^{k+1} q_y \| \\
\lesssim & o_L(1) (1 + \| \bar{q} \|_{X_{\langle k \rangle}}) (1 + \| q \|_{X_{\langle k \rangle}} )^2
\end{align*}

Next, we move to: 
\begin{align} \label{sub1.3}
\int_0^L |\mathcal{I}_{k+1, 4}^{(\theta)}| \le \sum_{j = 0}^{k+1} \binom{k+1}{j} \int_0^{L} |( \p_x^j \bar{u}_{yy} \p_x^{k+1 - j} q, \p_x^{k+1} q_{yy})|.
\end{align}
We split the above term into several cases. First, let us handle the $j = 0$ case for which (\ref{coe.2}) gives us the required bound by Hardy's inequality : 
\begin{align*}
|(\ref{sub1.3})[j=0]| \lesssim& \| \frac{1}{\bar{u}} \bar{u}_{yy} y  \|_\infty \| \p_x^{k} q_{xy} \| \| \sqrt{\bar{u}} \p_x^k q_{xyy} \| \lesssim o_L(1) p_{\langle k \rangle} \| q \|_{X_{k}}^2. 
\end{align*}

We now handle the case of $1 \le j \le k/2$, which requires a localization using $\chi$ as defined in (\ref{basic.cutoff}). We invoke \eqref{int:2} for the $\bar{v}$ contribution and \eqref{int:4}:
\begin{align*}
|(\ref{sub1.3})[1-\chi]| \lesssim &\| \p_x^{\frac{k}{2}} \bar{v}_{yyy} \|_{L^\infty_x L^2_y} \| \p_x^{\langle k \rangle} q \|_{L^2_x L^\infty_y} \| q \|_{X_k} \lesssim  o_L(1) p_{\langle k \rangle}  \| q \|_{X_{\langle k \rangle}}^2.
\end{align*}

  For the localized component, we integrate by parts in $y$ and invoking Lemma \ref{lemma:Sobolev}: 
\begin{align} \n
(\ref{sub1.3})[\chi] = & \sum_{j = 0}^{k+1} \int_{0}^{x_0} ( \p_x^k q_y \chi', \p_x^j \bar{u}_{yy} \p_x^{k+1-j}q ) +  \sum_{j = 0}^{k+1} \int_0^{x_0} ( \p_x^k q_y \chi, \p_x^j \bar{u}_{yyy} \p_x^{k+1-j} q ) \\ \label{fetty}
& +  \sum_{j = 0}^{k+1} \int_0^{x_0} \langle \p_x^k q_y \chi, \p_x^j \bar{u}_{yy} \p_x^{k+1-j} q_y \rangle \\ \n
\lesssim & \| \p_x^{k-1} q_{xy} \| \| \p_x^{\frac{k}{2}} \bar{v}_{yyy} \| \| \p_x^{\langle k \rangle} q \|_{L^\infty_{loc}}  + \| \p_x^{k-1} q_{xy} \| \| \p_x^{\frac{k}{2}} \bar{v}_{yyyy} \| \| \p_x^{\langle k \rangle} q \|_{L^\infty_{loc}} \\ \n
& + \| \p_x^{k-1} q_{xy} \| \| \p_x^{\frac{k}{2}} \bar{v}_{yyy} \|_\infty \| \p_x^{\langle k \rangle} q_{xy} \| \\  \n
\lesssim & o_L(1)p_{\langle k \rangle} \| q \|_{X_{\langle k \rangle}}^2.  
\end{align}

We now treat the case in which $k/2 \le j \le k$, which still requires localization and \eqref{int:4} and \eqref{int:2}
\begin{align*}
|(\ref{sub1.3})[\chi_{\ge 1}]| \lesssim  &\| \p_x^k q_{xyy} \sqrt{\bar{u}} \| \| \p_x^{\frac{k}{2}} q \|_{L^2_x L^\infty_y} \|\p_x^{\langle k - 1 \rangle} \bar{v}_{yyy} \|_{L^\infty_x L^2_y} \\
\lesssim & o_L(1) p_{\langle k \rangle} \| q \|_{X_{\langle k \rangle}}^2.
\end{align*}

Finally, we deal with the case when $y \lesssim 1$ for $k/2 \le j \le k$, which again requires integration by parts in $y$ as in (\ref{fetty}), upon invoking \eqref{int:1}, \eqref{int:2}
\begin{align*}
|(\ref{sub1.3}[j \ge k/2])| \lesssim &\| \p_x^{k-1} q_{xy} \| \| \p_x^{\langle k - 1 \rangle} \bar{v}_{yyy} \| \| \p_x^{\frac{k}{2}} q \|_{L^\infty_{loc}} \\
& + \| \p_x^{k-1} q_{xy} \| \| \p_x^{\langle k - 1 \rangle} \bar{v}_{yyyy} \| \| \p_x^{\frac{k}{2}} q \|_{L^\infty_{loc}} \\
& + \| \p_x^{k-1} q_{xy} \| \| \p_x^{\langle k - 1 \rangle}\bar{v}_{yyy} \|_\infty \| \p_x^{\langle k/2 \rangle} q_{xy} \| \\
\lesssim & o_L(1)p_{\langle k \rangle}(1 + \| q \|_{X_{\langle k \rangle}}^2). 
\end{align*}

We now move to the $\Lambda$ terms. By Hardy's inequality in $y$, we obtain 
\begin{align*}
\kappa \int_0^L |\mathcal{I}_{k+1,6}^{(\theta)}| \le & \kappa \sum_{j = 0}^{k+1} \binom{k+1}{j} \int_0^{x_0}(  \p_x^{j} \bar{v}_{xyy} \p_x^{k+1-j}I_x[v_y] , \p_x^{k+1} q) \\
 \lesssim &\kappa \| \p_x^{\langle \frac{k}{2} \rangle} \bar{v}_{xyy} \langle y \rangle \|_\infty \| \p_x^{\langle k+1 \rangle} I_x[v_y] \| \| \p_x^{k+1} q_y \| \\
& + \kappa \| \p_x^{\langle k+1 \rangle} \bar{v}_{xyy}  \| \| \p_x^{\langle \frac{k+1}{2} \rangle} I_x[v_y] \langle y \rangle \|_\infty \| \p_x^{k+1} q_y \|. 
\end{align*}
Next, we similarly have
\begin{align*}
\kappa \int_0^L |\mathcal{I}_{k+1,7}^{(\theta)}| \le& \kappa  \sum_{j = 0}^{k+1} \binom{k+1}{j} \int_0^{L}| (  \p_x^j \bar{v}_{yy} \p_x^{k+1-j} v_y, \p_x^{k+1} q  )| \\
\lesssim &\kappa  \| \p_x^{\langle \frac{k+1}{2} \rangle} \bar{v}_{yy} \langle y \rangle \|_\infty \| \p_x^{\langle k+1 \rangle} v_y \langle y \rangle \| \| \p_x^{k+1} q_y \| . 
\end{align*}
Next, we have 
\begin{align*}
\kappa  \int_0^L |\mathcal{I}_{k+1,8}^{(\theta)}| \le \kappa  \sum_{j = 0}^{k+1} \binom{k+1}{j} & \int_0^{L}| ( \p_x^j \bar{v}_{x} \p_x^{k+1-j} I_x[\p_y^3 v], \p_x^{k+1} q )| \\
 \le &\kappa \sum_{j = 0}^{k+1} \binom{k+1}{j} \int_0^{L}| ( \p_x^j \bar{v}_{xy} \p_x^{k+1-j} I_x[\p_y^2 v], \p_x^{k+1} q )| \\
 & +\kappa  \sum_{j = 0}^{k+1} \binom{k+1}{j} \int_0^{L} |( \p_x^j \bar{v}_{x} \p_x^{k+1-j} I_x[\p_y^2 v], \p_x^{k+1} q_y )|\\
\lesssim &\kappa  \| \p_x^{\langle \frac{k+1}{2} \rangle} \bar{v}_{xy} \langle y \rangle \|_\infty \| \p_x^{\langle k+1 \rangle} I_x[v_{yy}] \| \| \p_x^{k+1} q_y \| \\
& + \kappa  \| \p_x^{\langle \frac{k+1}{2}\rangle} \bar{v}_{x} \|_\infty \| \p_x^{\langle k + 1 \rangle} I_x[v_{yy}] \| \| \p_x^{k+1} q_{y} \| \\
& + \kappa \| \p_x^{\langle \frac{k+1}{2} \rangle} I_x[v_{yy}] \langle y \rangle \|_\infty \| \p_x^{\langle k + 1 \rangle} \bar{v}_{xy} \| \| \p_x^{k+1} q_y\| \\
& +\kappa  \| \p_x^{\langle k +1 \rangle} \bar{v}_{xy} \| \| \p_x^{\langle \frac{k+1}{2} \rangle} I_x[v_{yy}] \langle y \rangle \|_\infty \| \p_x^{k+1} q_{y} \|.
\end{align*}

\begin{align*}
\kappa\int_0^L |\mathcal{I}_{k+1,9}^{(\theta)}| \le &\kappa \sum_{j = 0}^{k+1} \binom{k+1}{j} \int_0^{L} |( \p_x^j \bar{v} \p_x^{k+1 - j} v_{yyy}, \p_x^{k+1} q )| \\
\le & \kappa \sum_{j = 0}^{k+1} \binom{k+1}{j} \Big(  \int_0^{L} | ( \p_x^{k+1 - j} v_{yy} \p_x^j \bar{v}_{y}, \p_x^{k+1} q )| +   \int_0^{L}| ( \p_x^{k+1 - j} v_{yy} \p_x^j \bar{v}, \p_x^{k+1} q_y )| \Big) \\
 \lesssim & \kappa\| \p_x^{\langle \frac{k+1}{2} \rangle} \bar{v}_{y} \langle y \rangle \|_\infty \| \p_x^{\langle k + 1 \rangle} v_{yy} \| \| \p_x^{k+1} q_y \|  + \kappa\| \p_x^{\langle \frac{k+1}{2} \rangle} \bar{v}_{} \|_\infty \| \p_x^{\langle k + 1 \rangle} v_{yy} \| \| \p_x^{k+1} q_{y} \| \\
& +\kappa \| \p_x^{\langle \frac{k+1}{2} \rangle} v_{yy} \langle y \rangle \|_\infty \| \p_x^{\langle k+1 \rangle} \bar{v}_{y} \| \| \p_x^{k+1} q_y \|  + \kappa \| \p_x^{\langle \frac{k+1}{2} \rangle} \bar{v} \|_\infty \| \p_x^{\langle k + 1 \rangle} v_{yy} \| \| \p_x^{k} q_{xy} \|.
\end{align*}

\noindent  Summarizing the $\Lambda$ contributions by repeated use of Lemma \ref{lemma:Sobolev} with $k+1$ for the $\bar{v}$ contribution gives: 
\begin{align*}
&\sup_{x_0 \le L} \int_0^{x_0} (\p_x \p_x^k \Lambda, \p_x \p_x^k q )  \lesssim \kappa o_L(1)p_{\langle k + 1 \rangle} (1 + \| q \|_{X_{\langle k \rangle}}^2 ).
\end{align*}

Next, we have the contributions at $y = 0$: 
\begin{align*}
\int_0^L |\mathcal{I}_{k+1,10}^{(\theta)}| \le &\sum_{j = 1}^{k+1} \binom{k+1}{j}  \int_0^{L} |\p_x^j \bar{u}_y \p_x^{k+1-j} q_y(0) \p_x^{k+1} q_y(0) | \\
 \le &o_L(1)(1 + p_{\langle k \rangle})\| \p_x^{k+1-j} q_y(0) \|_{L^2_x} \| \p_x^k q_y(0) \|_{L^2_x} \le  o_L(1) (1 + p_{\langle k \rangle})\| q \|_{X_{\langle k  \rangle}}^2. 
\end{align*}
Above, we have used that $\p_x^j \bar{u}_y(x, 0) = - \p_x^{j-1} \bar{v}_{yy}(x, 0)$, and then \eqref{int:2}. 

Finally, we have the $u^0$ contributions for which we rely upon \eqref{int:2} for $k + 1$ for the $\bar{v}$ term: 
\begin{align*}
\kappa\int_0^L |\mathcal{I}_{k+1,5}^{(\theta)}|& =\kappa \int_0^{L} |( \p_x^{k+1} U, \p_x^{k} q_x )|  =\kappa \int_0^{L} ( u^0 \p_x^{k+1} \bar{v}_{xyy} - u^0_{yy} \p_x^{k+1} \bar{v}_{x}, \p_x^{k} q_x ) \\
& \hspace{10 mm} \lesssim \kappa o_L(1) \| u^0, u^0_{yy} \cdot \langle y \rangle \|_\infty \| \p_x^{k+1} \bar{v}_{xyy} \| \| \p_x^{k} q_{xy} \| \\
& \hspace{10 mm} \lesssim \kappa o_L(1) p_{\langle k + 1 \rangle} \| q \|_{X_{\langle k \rangle}}. 
\end{align*}
This concludes the proof of the lemma. 
\end{proof}

\begin{lemma}[$\p_x^k$ $\p_y^4$ Estimate] \label{locvyyyy}   Assume $v$ is a solution to (\ref{origPrLay.beta}). Then the following estimate holds: 
\begin{align}  \label{solid.py4.cof}
\begin{aligned}
\| \p_x^k v_{yyyy} \|_{2,loc} \lesssim & \| q \|_{\mathcal{E}_{\langle k \rangle}} + o_L(1)( p_{\langle k  \rangle} + \kappa p_{\langle k + 1 \rangle}) (1 + \| q \|_{X_{\langle k \rangle}} ) \\
& +\kappa  o_L(1) C(u^0) + \| \p_{xy} \p_x^k g_1 \|_{2,loc}.
\end{aligned}
\end{align}
\end{lemma}

\begin{proof} We apply $\p_x^k$ to the equation \eqref{origPrLay.beta} to obtain the following pointwise inequality:
\begin{align}
\begin{aligned} \label{sasak}
|v^{(k)}_{yyyy}| \lesssim &|\bar{u}^{j_1} \bar{u}^{j_2} q^{(k-j)}_{xyy}| + |\bar{u}_{x}^{j_1} \bar{u}_{y}^{j_2} q_y^{k-j}| + |\bar{u}^{j_1} \bar{u}_{xy}^{j_2} q_y^{k-j}| + |\bar{u}^{j_1} \bar{u}_{y}^{j_2} q_{xy}^{k-j}| \\
& + |\bar{u}^{j_1} \bar{u}_{x}^{j_2} q^{(k-j)}_{yy}| + \kappa |\bar{v}^j_{xyy} I_x[v^{(k-j)}_y]| +  \kappa |\bar{v}^j_{yy}v^{(k-j)}_y| + \kappa |v^{j}_{sx} I_x[v^{(k-j)}_{yyy}]| \\
& + |\bar{v}^j v^{(k-j)}_{yyy} | +  \kappa |v^{(k)}_{sxyy} u^0| +  \kappa |v^{(k)}_{sx} u^0_{yy}| + |\p_{xy}g_1^k|.
\end{aligned}
\end{align}

Placing the terms on the right-hand side above in $L^2_{loc}$ gives the desired result by applying Lemma \ref{lemma:Sobolev} with $k$ and $k+1$: 
\begin{align*}
&\| (\ref{sasak}.1)\| \lesssim \| \bar{v}_{y}^{\langle k - 1 \rangle}, \bar{v}_{yy}^{\langle k - 1 \rangle} \|_\infty^2 \| \bar{u} q_{xyy}^{\langle k \rangle} \|, \\
&\| (\ref{sasak}.2) \| \lesssim \| \bar{v}_{y}^{\langle k - 1 \rangle} \|_\infty \| \bar{v}_{yy}^{\langle k - 1 \rangle} \|_\infty \| q^{\langle k \rangle}_y \| \\
&\| (\ref{sasak}.3) \| \lesssim \| \bar{v}_{y}^{\langle k - 1 \rangle} \|_\infty \| \bar{v}_{yy}^{\langle k - 1 \rangle} \|_\infty \| q^{\langle k \rangle}_y \| \\
&\| (\ref{sasak}.4) \| \lesssim \| \bar{v}^{\langle k - 1 \rangle}_{y} \|_\infty \| \bar{v}^{\langle k - 1 \rangle}_{yy} \|_\infty \| q^{\langle k \rangle}_{xy} \| \\
&\| (\ref{sasak}.5) \| \lesssim \| \bar{v}_{y}^{\langle k - 1 \rangle} \|_\infty \| \bar{v}^{\langle k \rangle}_{y} \|_\infty \| q^{\langle k \rangle}_{yy} \| \\
&\| (\ref{sasak}.6, 7) \| \lesssim \kappa \| \bar{v}^{\langle k + 1 \rangle}_{yy} \| \| v_y^{\langle k \rangle} \|_\infty \\
&\| (\ref{sasak}.8, 9) \| \lesssim (\| \bar{v}^{\langle k \rangle} \|_\infty + \kappa \|\bar{v}^{\langle k + 1 \rangle} \|_\infty) \| v^{\langle k \rangle}_{yyy} \| \\
&\| (\ref{sasak}.10, 11) \| \lesssim  \kappa \| u^0, u^0_{yy} \langle y \rangle^2 \|_\infty \| \bar{v}^{k+1}_{yy} \|
\end{align*}
These estimates immediately establish the proof of the lemma. 
\end{proof}

We now move to a $\| \cdot \|_{\mathcal{H}_k}$ estimate, for which we first recall the definition in (\ref{norm.X.layer}).
\begin{lemma}[Weighted $\p_x^k H^4$]  \label{Hkestimate:8338}  Assume $q$ solves (\ref{origPrLay.beta}). Then the following estimate is valid:
\begin{align}
\begin{aligned}
\| q \|_{\mathcal{H}_k}^2 \lesssim & \| \p_x^k \p_{xy} g_1 \cdot w \chi \|^2 + C_k(q_0) + \kappa o_L(1) C(u^0) \\ \label{weight.H4.cof}
& +( o_L(1) p_{\langle k \rangle} + \kappa o_L(1) p_{\langle k + 1 \rangle})\Big(1+ \| q \|_{X_{\langle k \rangle}}^2 \Big).
\end{aligned}
\end{align}
\end{lemma}
\begin{proof} We take $\p_x^k$ of equation (\ref{theta:approx:1}), which produces: 
\begin{align}
\begin{aligned} \label{produce}
&- \p_{xy} ( \bar{u}_\theta^2 q^{(k)}_y ) + v^{(k)}_{yyyy}  = \p_x^k \p_{xy}g_1 -  \kappa \p_x^k \Lambda(v) - \kappa \p_x^k U \\
&\qquad \qquad -  \p_{xy} (\sum_{j = 1}^k \sum_{j_1 + j_2 = j} c_{j_1, j_2, j} \p_x^{j_1}\bar{u}_\theta \p_x^{j_2} \bar{u}_\theta  q^{(k-j)}_y ). 
\end{aligned}
\end{align}
We now square and integrate both sides in $y$, for frozen $x = x_0$, against the weight $w (1 - \chi)$, which produces the inequality  
\begin{align} \n
&\| (- \p_{xy} ( \bar{u}_\theta^2 q^{(k)}_y ) + v^{(k)}_{yyyy} ) w (1 - \chi) \|_{L^2_y}^2 \\ \n
\le& \| \p_x^k \p_{xy}g_1 w (1-\chi) \|_{L^2_y}^2 + \kappa \| \p_x^k \Lambda(v) w (1-\chi) \|_{L^2_y}^2 +\kappa  \| \p_x^k U w (1-\chi)\|_{L^2_y}^2 \\ \label{produce:produce}
& + \|  \p_{xy} (\sum_{j = 1}^k \sum_{j_1 + j_2 = j} c_{j_1, j_2, j} \p_x^{j_1}\bar{u}_\theta \p_x^{j_2} \bar{u}_\theta  q^{(k-j)}_y ) w (1-\chi)\|_{L^2_y}^2. 
\end{align}

We start by expanding out the terms on the left-hand side, which gives
\begin{align} \n
&\|[\p_{xy} \{\bar{u}_\theta^2 q^{(k)}_y \} - v^{(k)}_{yyyy}] \cdot w \{1 - \chi\}\|_{x = x_0}^2 \\ \n
 \ge&  \| v^{(k)}_{yyyy} \{1 - \chi \} w \|_{x = x_0}^2 +  \| [ \bar{u}_\theta^2 q^{(k)}_{xyy}  \{1 - \chi \} w \|_{x = x_0}^2 - 4 \| \bar{u}_\theta \bar{u}_{\theta x} q^{(k)}_{yy} w(1-\chi) \|_{L^2_y}^2 \\ \n
& - 4 \| \bar{u}_{\theta y} \bar{u}_{\theta x} q^{(k)}_y w(1-\chi) \|_{L^2_y}^2 - 4 \| \bar{u} \bar{u}_{\theta x y} q^{(k)}_y w (1 - \chi) \|_{L^2_y}^2 - 4 \| \bar{u}_\theta \bar{u}_{\theta y} q^{(k)}_{xy} w(1-\chi) \|_{L^2_y}^2 \\ \n
& - (2 v^{(k)}_{yyyy}, \bar{u}_\theta^2 q^{(k)}_{xyy} w^2 (1 - \chi)^2) - 4( v^{(k)}_{yyyy}, \bar{u}_\theta \bar{u}_{\theta x} q^{(k)}_{yy} w^2 (1 - \chi)^2) \\ \n
& - 4(v^{(k)}_{yyyy}, \bar{u}_{\theta y} \bar{u}_{\theta x} q^{(k)}_y w^2 (1 - \chi)^2) - 4 (v^{(k)}_{yyyy}, \bar{u}_\theta \bar{u}_{\theta x y} q^{(k)}_y w^2 (1-\chi)^2)\\ \n
& - 4 (v^{(k)}_{yyyy}, \bar{u}_\theta \bar{u}_{\theta y} q^{(k)}_{xy} w^2 (1 - \chi)^2) =: \mathcal{P}_1 + ... + \mathcal{P}_{11}.
\end{align}

\noindent  All terms are estimated in a straightforward manner except for $\mathcal{P}_7$, upon invoking $\bar{u}_x + \bar{v}_y = 0$, so we begin with: 
\begin{align*}
&|\mathcal{P}_3| \lesssim \| \bar{u}_{x}  \|_\infty^2 \|q^{(k)}_{yy} \cdot w \|_{x = x_0}^2 \\
&|\mathcal{P}_4| \lesssim \| \bar{u}_{y} \|_\infty^2 \| \bar{v}_{y} \|_\infty^2 \|q^{(k)}_y w \|_{x = x_0}^2 \\
&|\mathcal{P}_5| \lesssim \| \bar{u} \bar{v}_{yy} \|_\infty^2 \|q^{(k)}_y w \|_{x = x_0}^2,  \\
&|\mathcal{P}_6| \lesssim \| \bar{u} \bar{u}_{y} w \|_\infty^2 \| q^{(k)}_{xy} \|_{x = x_0}^2, \\
&|\mathcal{P}_8| \lesssim \| \bar{u}_{x} \|_\infty | v^{(k)}_{yyyy} w \{1 -\chi\} \|_{x = x_0} \| q^{(k)}_{yy} w \{1 -\chi\}\|_{x = x_0} \\
&|\mathcal{P}_9| \lesssim \| \bar{u}_{y} \bar{v}_{y}  \|_\infty |q^{(k)}_y w \{1 -\chi\}\|_{x = x_0} \|v^{(k)}_{yyyy} w \{1 -\chi\}\|_{x = x_0}, \\
&|\mathcal{P}_{10}| \lesssim \| \bar{u} \bar{v}_{yy} \|_\infty |q^{(k)}_y w  \{1 -\chi\}\|_{x = x_0} \|v^{(k)}_{yyyy}w \{1 -\chi\}\|_{x = x_0} \\
&|\mathcal{P}_{11}| \lesssim \| \bar{u} \bar{u}_{y} w \|_\infty \| q^{(k)}_{xy} \|_{x = x_0} \|v^{(k)}_{yyyy} w \{1 -\chi\}\|_{x = x_0}.
\end{align*}

\noindent  Upon integrating in $x$ and applying Lemma \ref{lemma:Sobolev}, we may summarize the above estimates via: 
\begin{align*}
&|\mathcal{P}_3| + ... +  |\mathcal{P}_6| + |\mathcal{P}_8| +... + |\mathcal{P}_{11}| \lesssim o_L(1) p(\| \bar{q} \|_{X_{\langle 1 \rangle}}) (1 + \| q \|_{X_k} ). 
\end{align*}

We thus move to $\mathcal{P}_7$ for which we integrate by parts once in $y$ to obtain  
\begin{align*}
\mathcal{P}_7 = & 2(v^{(k)}_{yyy}, \bar{u}_\theta^2 q^{(k)}_{xyyy} w^2(1 - \chi)^2) + 4(v^{(k)}_{yyy}, \bar{u}_\theta \bar{u}_{\theta y} q^{(k)}_{xyy} w^2 (1 - \chi)^2) \\
& + 4(v^{(k)}_{yyy}, \bar{u}_\theta^2 w w_y (1 - \chi)^2) - 4(v^{(k)}_{yyy}, \bar{u}_\theta^2 w^2 (1 - \chi) \chi') =: \mathcal{P}_{7,1} + \dots + \mathcal{P}_{7,4}.
\end{align*}
We now expand using that $v = \bar{u} q$,  
\begin{align*}
v^{(k)}_{yyy} := & \p_x^k \{ \bar{u} q \}_{yyy} = \p_x^k \{ \bar{u}_{yyy}q + 3 \bar{u}_{yy} q_y + 3 \bar{u}_y q_{yy} + \bar{u} q_{yyy} \} \\
= & \sum_{j = 0}^k c_j \p_x^j \bar{u}_{yyy} \p_x^{k-j} q + 3 c_j \p_x^j \bar{u}_{yy} \p_x^{k-j} q_y + 3 c_j \p_x^j \bar{u}_y \p_x^{k-j} q_{yy} + c_j \p_x^j \bar{u} \p_x^{k-j} q_{yyy}. 
\end{align*}

\noindent First, upon integrating by parts once in $y$ (ignoring commutator terms, which are dealt with below), let us highlight the main positive contribution from the last term above, for $j = 0$:
\begin{align*}
 &2(\bar{u} \p_x^k q_{yyy}, \bar{u}_\theta^2 q^{(k)}_{xyyy} \{1 - \chi \}^2 w^2)_{x = x_0} = (2 \bar{u} \bar{u}_\theta^2 q^{(k)}_{yyy}, q^{(k)}_{xyyy} \{1 - \chi\}^2 w^2)_{x = x_0} \\
 = & \p_x \| |\bar{u}|^{\frac{1}{2}} \bar{u}_\theta q^{(k)}_{yyy} \{1 - \chi \} w \|_{x =x_0}^2 - ( (\bar{u} \bar{u}_\theta^2)_x q^{(k)}_{yyy}, q^{(k)}_{yyy} \{1 - \chi \}^2 w^2)_{x = x_0}.
\end{align*} 

\noindent  Hence:
\begin{align} \n 
- 2( v^{(k)}_{yyyy}, &\bar{u}_\theta^2 q^{(k)}_{xyy} w^2 \{1 - \chi\}^2)_{x = x_0} \\ \n
= &\p_x \| q^{(k)}_{yyy} |\bar{u}|^{\frac{3}{2}} w \{1 - \chi\} \|_{x = x_0}^2 -  (|q^{(k)}_{yyy}|^2,  (\bar{u} \bar{u}_\theta^2)_{x} w^2 \{1 - \chi\}^2)_{x = x_0} \\ \n
& - ( q_{xyy}^k, \p_y \{ [ \bar{u}_{yyy}^j q^{(k-j)} + 3 \bar{u}_{yy}^j q_y^{k-j} + 3 \bar{u}_{y}^j q_{yy}^{k-j} ] \\  \n
& \times \bar{u}_\theta^2 w^2 \{1 - \chi\}^2 \}) + ( 2 v^{(k)}_{yyy}, q^{(k)}_{xyy} \p_y \{ \bar{u}_\theta^2 w^2 \{1 - \chi\}^2 \})_{x = x_0} \\ \ \n
& + \sum_{j = 1}^k c_j (\p_x^j \bar{u} \p_x^{k-j} q_{yyy}, \bar{u}_\theta^2 q^{(k)}_{xyyy} w^2 \{1 - \chi \}^2)_{x = x_0} \\ \label{belly:of:june}
= & \p_x \| q^{(k)}_{yyy} |\bar{u}|^{\frac{3}{2}} w \{1 - \chi\} \|_{x = x_0}^2 + \Upsilon_1 + \dots + \Upsilon_6.
\end{align}

First, we estimate upon integrating from $x = 0$ to $x = x_0$, 
\begin{align*}
|\int_0^{x_0} \Upsilon_1(x) \ud x| \lesssim o_L(1) \sup_x \| |\bar{u}|^{\frac{3}{2}} q^{(k)}_{yyy} w\{1- \chi\} \|_{x = x_0}^2 \lesssim o_L(1) \| q \|_{X_k}^2. 
\end{align*}

Next, 
\begin{align*}
\Upsilon_2 = & - ( q^{(k)}_{xyy} \bar{u}^j_{yyyy},  q^{(k-j)} \bar{u}_\theta^2 w^2 \{1 - \chi\}^2)_{x = x_0} - ( q^{(k)}_{xyy} \bar{u}^j_{yyy}, q^{(k-j)} \bar{u}_\theta \bar{u}_{\theta y} w^2 \{1 - \chi\}^2)_{x = x_0} \\
& - ( q^{(k)}_{xyy} \bar{u}^j_{yyy}, q^{(k-j)}_y \bar{u}_\theta^2 w^2 \{1 - \chi\}^2)_{x = x_0} - ( q^{(k)}_{xyy} \bar{u}^j_{yyy}, q^{(k-j)} \bar{u}_\theta^2 2 ww_{y} \{1 - \chi\}^2)_{x = x_0} \\
& - ( q^{(k)}_{xyy} \bar{u}^j_{yyy}, q^{(k-j)} \bar{u}_\theta^2 w^2 \{1 - \chi\} \chi')_{x = x_0} =: \Upsilon_{2,1} + \dots + \Upsilon_{2,5}.
\end{align*}

\noindent  We will estimate each term above with the help of the Prandtl identities, which follow from (\ref{Pr.leading}), for $\bar{u}$:
\begin{align}
\begin{aligned} \label{pink}
&|\bar{u}^j_{yyy}| \lesssim |\p_x^j \{ \bar{u} \bar{v}_{yy} + \bar{v} \bar{u}_{yy} \}|, \\
&|\bar{u}^j_{yyyy}| \lesssim |\p_x^j \{ \bar{u}_{y} \bar{v}_{yy} + \bar{u} \bar{v}_{yyy} + \bar{v}_{y} \bar{u}_{yy} + \bar{u} \bar{v} \bar{v}_{yy} + \bar{v}^2 \bar{u}_{yy} \}|
\end{aligned}
\end{align}
Inserting this expansion into term $\Upsilon_{2,1}$, gives 
\begin{align*}
|\Upsilon_{2,1}| \lesssim & |( \p_x^j(\bar{u}_y \bar{v}_{yy})  q^{(k)}_{xyy} ,  q^{(k-j)} \bar{u}_\theta^2 w^2 \{1 - \chi\}^2)| + |( \p_x^j(\bar{u} \bar{v}_{yyy})  q^{(k)}_{xyy} ,  q^{(k-j)} \bar{u}_\theta^2 w^2 \{1 - \chi\}^2)| \\
& + |( \p_x^j(\bar{u}_{yy} \bar{v}_{y})  q^{(k)}_{xyy} ,  q^{(k-j)} \bar{u}_\theta^2 w^2 \{1 - \chi\}^2)| + |( \p_x^j(\bar{u} \bar{v} \bar{v}_{yy})  q^{(k)}_{xyy} ,  q^{(k-j)} \bar{u}_\theta^2 w^2 \{1 - \chi\}^2)|\\ 
 &+ |( \p_x^j(\bar{v}^2 \bar{u}_{yy})  q^{(k)}_{xyy} ,  q^{(k-j)} \bar{u}_\theta^2 w^2 \{1 - \chi\}^2)| = : \Upsilon_{2,1,1} + \dots + \Upsilon_{2,1,5}. 
\end{align*}

First, upon integration in $x$ and using \eqref{int:3}, \eqref{int:4}
\begin{align*}
\int_0^{L}| \Upsilon_{2,1,1} | \lesssim & \int_0^{L} | (q^{(k)}_{xyy}, q^{\langle k  \rangle} \bar{u}_\theta^2 w^2 \{1 - \chi\}^2 \bar{u}^{\langle k \rangle}_{y} \bar{v}^{\langle k \rangle}_{yy})| \\
\lesssim & \int_0^{L} \|q^{(k)}_{xyy} \{1 - \chi\} w\|_{x = x_0} \|q^{\langle k  \rangle}\|_\infty \|\bar{v}^{\langle k -1 \rangle}_{yy} \|_\infty \|\bar{v}^{\langle k \rangle}_{yy} w  \{1 - \chi\}\|_{x = x_0}, \\
\lesssim & (1 + o_L(1) p_k)^2 \| q \|_{X_k} o_L(1) \| 1 + \| q \|_{X_k}). 
\end{align*}

\noindent Next, upon invoking \eqref{int:2} and \eqref{int:4}, 
\begin{align*}
\int_0^{L}| \Upsilon_{2,1,2}| \lesssim & |\int_0^{L}( q^{(k)}_{xyy}, q^{\langle k  \rangle} \bar{u}_\theta^2 w^2 \{1 - \chi\}^2 \p_x^{\langle k \rangle} \bar{u} \p_x^{\langle k \rangle} \bar{v}_{yyy})| \\
\lesssim &  \|q^{\langle k  \rangle}\|_\infty \|\p_x^{\langle k -1  \rangle} \bar{v}_{y} \|_\infty \|q^{(k)}_{xyy} \{1 - \chi\} w\|  \|\p_x^{\langle k  \rangle} \bar{v}_{yyy} \{1 - \chi\} w\| \\
\lesssim & (1 + o_L(1) p_k)^2 \| q \|_{X_k} o_L(1) (1 + \| q \|_{X_k}).
\end{align*}

\noindent Next, upon invoking \eqref{int:2}, \eqref{int:4}, \eqref{int:5}, we have 
\begin{align*}
\int_0^{L} |\Upsilon_{2,1,3} | \lesssim &\int_0^{x_0}|( q^{(k)}_{xyy}, q^{\langle k  \rangle} \p_x^{\langle k \rangle} \bar{v}_{y} \p_x^{\langle k \rangle} \bar{u}_{yy} \{1 - \chi\}^2 w^2 \bar{u}_\theta^2)| \\
\lesssim & \|q^{(k)}_{xyy} \{1 - \chi\} w\|  \|q^{\langle k  \rangle}\|_\infty \|\p_x^{\langle k - 1 \rangle} \bar{v}_{yyy} \langle y \rangle^2\|_\infty \|\p_x^{\langle k \rangle} \bar{v}_{y}  w \{1 - \chi\}\|, \\
\lesssim & (1 + o_L(1) p_k)^2 \| q \|_{X_k} o_L(1) (1 + \| q \|_{X_k}). 
\end{align*}

\noindent Next, upon invoking \eqref{int:2}, \eqref{int:4}, we have 
\begin{align*}
\int_0^{L} | \Upsilon_{2,1,4} | \lesssim & \int_0^{L} |( q^{(k)}_{xyy}, q^{\langle k  \rangle} \bar{u}_\theta^2 w^2 \{1 - \chi\}^2 \p_x^{\langle k - 1 \rangle} \bar{v}_{y} \p_x^{\langle k \rangle} \bar{v} \p_x^{\langle k \rangle} \bar{v}_{yy})| \\
\lesssim & \|q^{(k)}_{xyy} \{1 - \chi\} w\| \|q^{\langle k  \rangle}\|_\infty \|\p_x^{\langle k \rangle}\bar{v}\|_\infty \|\p_x^{\langle k - 1 \rangle} \bar{v}_{y}\|_\infty \|\p_x^{\langle k \rangle} \bar{v}_{yy} w \{1 - \chi\}\|, \\
\lesssim & ( 1+ o_L(1) p(\| q \|_{X_k}) ) (1 + o_L(1) p_k)^2 \| q \|_{X_k} o_L(1) ( 1+ \| \bar{q} \|_{X_k}). 
\end{align*}

\noindent Next, upon invoking \eqref{int:4}
\begin{align*}
\int_0^{L} |\Upsilon_{2,1,5}| \lesssim & |\int_0^{L} ( q^{(k)}_{xyy}, q^{\langle k \rangle} \p_x^{\langle k \rangle} \bar{v} \p_x^{\langle k \rangle}\bar{v} \p_x^{\langle k - 1 \rangle} \bar{v}_{yyy} \bar{u}_\theta^2 \{1 - \chi\}^2 w^2)| \\
\lesssim & \|q^{(k)}_{xyy} \{1 - \chi\} w\| \|q^{\langle k \rangle}\|_\infty \|\p_x^{\langle k \rangle} \bar{v}\|_\infty^2 \|\p_x^{\langle k - 1 \rangle}\bar{v}_{yyy} \{1 - \chi\} w \| \\
\lesssim & \| q \|_{X_k} (1 + o_L(1) p(\| q \|_{X_k})) (1 + o_L(1) p_k) o_L(1) (1 + p_k) \\
\lesssim & o_L(1) + o_L(1) p_{\langle k \rangle} + o_L(1) \| q \|_{X_{\langle k \rangle}}^2. 
\end{align*}

We now move to: 
\begin{align*}
\int_0^{L} | \Upsilon_{2,2}| \lesssim & \int_0^{x_0} | (q^{(k)}_{xyy}, q^{\langle k \rangle} w^2 \bar{u} \bar{u}_{y} \{1 - \chi\}^2 [ \bar{u}^{\langle k \rangle} \bar{v}^{\langle k \rangle}_{yy} + v^{\langle k \rangle}_s \bar{u}^{\langle k \rangle}_{yy}] )| \\
\lesssim & \|q^{(k)}_{xyy} w \{1 - \chi\}\|  \|q^{\langle k  \rangle}\|_\infty \Big[ \|\bar{v}^{\langle k - 1 \rangle}_{y}\|_\infty \|\bar{v}^{\langle k \rangle}_{yy} w \{1 - \chi\}\|  \\
& + \|\bar{v}^{\langle k \rangle}\|_\infty \|\bar{u}^{\langle k \rangle}_{yy} w \{1 - \chi\}\| \Big] \| \| \bar{u}_{y} \langle y \rangle \|_\infty \\
\lesssim & \| q \|_{X_k} (1 + o_L(1) \| q \|_{X_k})(1 + o_L(1) p_{k-1}) o_L(1) (1 + \| \bar{q} \|_{X_k}). 
\end{align*}

\noindent Above, we have invoked \eqref{int:2} and \eqref{int:4}.

Next, we use \eqref{pink} and Lemma \ref{lemma:Sobolev} to estimate, 
\begin{align}
\begin{aligned} \n
\int_0^{x_0} |\Upsilon_{2,3}| \le & \sum_{j = 0}^{k} \binom{k}{j} \int_x |(q^{(k)}_{xyy} \bar{u}^{j}_{yyy}, q^{(k-j)}_{y} \bar{u}^2 w^2 \{1 - \chi\}^2 )_{x = x_0} \ud x_0 \\
\lesssim & \int_x |(q^{(k)}_{xyy} \bar{u}_{yyy}, q^{(k)}_y \bar{u}^2 w^2 \{1 - \chi\}^2)_{x = x_0} \ud x_0 \\
& + \int_x |(q^{(k)}_{xyy} \bar{u}^{(k)}_{yyy}, q_y w^2 \{1 - \chi \}^2)_{x = x_0} \ud x_0| \\
& + \sum_{j= 1}^{k-1} \int_x |(q^{(k)}_{xyy} \bar{u}^{(j)}_{yyy}, q^{(k-j)}_y \bar{u}^2 w^2 \{1 - \chi \}^2)_{x = x_0}| \\
\lesssim & \| \bar{u}_{yyy} \|_\infty \| \bar{u} q^{(k)}_{xyy} w \{1 - \chi \} \| \| q^{(k)}_y w \{1 - \chi \} \|  \\
& + \| q^{(k)}_{xyy} w \{1 - \chi \} \| \| \bar{u}^{(k)}_{yyy} w \| \| q_y \|_\infty \\
& + \| \bar{v}_{yyyy}^{\langle k - 2 \rangle} w \|_{L^\infty_x L^2_y} \| q^{(k)}_{xyy} w \| \| q^{\langle k - 1 \rangle}_y \|_{L^2_x L^\infty_y} \\
\lesssim & o_L(1) (1 + o_L(1) p_k) \| q \|_{X_k}.
\end{aligned}
\end{align}

Next, again by invoking Lemma \ref{lemma:Sobolev}, we estimate 
\begin{align*}
\int_x  |\Upsilon_{2,4}| \lesssim & |(q^{(k)}_{xyy}, q^{\langle k \rangle} w^2 \{1 - \chi\}^2 \bar{u}^2 [\bar{u}^{\langle k \rangle} \bar{v}^{\langle k \rangle}_{yy} + \bar{v}^{\langle k \rangle} \bar{u}^{\langle k \rangle}_{yy} ])| \\
\lesssim & \|q^{(k)}_{xyy} w \{1 - \chi\}\|_{x = x_0} \|q^{\langle k  \rangle}\|_\infty \Big[ \|\bar{v}^{\langle k \rangle}_{yy} w \{1 - \chi\}\|_{x = x_0} \|\bar{v}^{\langle k - 1 \rangle}_{y} \|_\infty \\
&+ \|\bar{v}^{\langle k \rangle}\|_\infty \times   \|\bar{v}^{\langle k - 1 \rangle}_{yyy} w \{1 - \chi\}\|_{x = x_0}  \Big] \\
\lesssim & o_L(1) p_k \| q \|_{X_{\langle k \rangle}}^2, \\
\int_x  |\Upsilon_{2,5}| \lesssim & |( q^{(k)}_{xyy}, q^{\langle k \rangle} \{1 - \chi\} \{1 - \chi\}' [ \bar{u}^{\langle k \rangle} \bar{v}^{\langle k \rangle}_{yy} + \bar{v}^{\langle k \rangle} \bar{u}^{\langle k \rangle}_{yy} ])| \\
\lesssim & \|q^{(k)}_{xyy}\|_{loc} \|q^{\langle k  \rangle}\|_{\infty, loc} \Big[ \|\bar{v}_{y}^{\langle k - 1\rangle}\|_{\infty, loc} \|\bar{v}^{\langle k \rangle}_{yy}\|_{loc} \\
& + \|\bar{v}^{\langle k \rangle}\|_{loc} \|\bar{v}^{\langle k - 1 \rangle}_{yyy}\|_{\infty, loc} \Big] \\
\lesssim & o_L(1) p_k \| q \|_{X_{\langle k \rangle}}^2. 
\end{align*}
 
We now estimate the following:
\begin{align*}
\int_x |\Upsilon_3| \lesssim & |( q^{(k)}_{xyy} \bar{u}^j_{yyy}, q^{(k-j)}_y \bar{u}^2 w^2 \{1 - \chi\}^2)| + |( q^{(k)}_{xyy} \bar{u}^j_{yy}, q^{(k-j)}_{yy} \bar{u}^2 w^2 \{1 - \chi\}^2)| \\
& + |( q^{(k)}_{xyy}, \bar{u}^j_{yy} q^{(k-j)}_y \bar{u} \bar{u}_{y} w^2 \{1 - \chi\}^2)| + |(q^{(k)}_{xyy}, \bar{u}^j_{yy} q^{(k-j)}_y \bar{u}^2 ww_{y} \{1 - \chi\}^2)| \\
&  + |( q^{(k)}_{xyy}, \bar{u}^j_{yy} q^{(k-j)}_y \bar{u}^2 w^2 \{1 - \chi\} \{1 - \chi\}' )| \\
\lesssim & o_L(1) p_{k} \| q \|_{X_{\langle k \rangle}}^2. 
\end{align*}

\noindent  Above, we have used the estimates which proceed below, with the use of the identities (\ref{pink}) and Lemma \ref{lemma:Sobolev}:
\begin{align*}
\int_x |\Upsilon_{3,1}| \lesssim & |( q^{(k)}_{xyy}, q^{\langle k \rangle}_y \bar{u}^2 w^2 \{1 - \chi\}^2 [ \bar{u}^{\langle k \rangle} \bar{v}^{\langle k \rangle}_{yy} + \bar{v}^{\langle k \rangle} \bar{u}^{\langle k \rangle}_{yy} ] )| \\
\lesssim & \|q^{(k)}_{xyy} w \{1 - \chi\}\|_{x = x_0} |q^{\langle k  \rangle}_y w\{1 - \chi\}\|_{x = x_0} \Big[ \| \bar{v}_{y}^{\langle k - 1 \rangle}\|_\infty \| \bar{v}^{\langle k - 1 \rangle}_{yy}  \|_\infty \\
& + \| \bar{v}^{\langle k \rangle} \|_\infty  \| \bar{v}_{yy}^{\langle k- 1\rangle} \|_\infty \Big] \\
\int_x |\Upsilon_{3,2}| \lesssim & \|q^{(k)}_{xyy} w \{1 - \chi\}\|_{x = x_0} |q^{\langle k  \rangle} \{1 - \chi\} w \langle y \rangle^{-1}\|_{x = x_0} \| \bar{v}_{yyy}^{\langle k - 1 \rangle} \langle y \rangle \|_\infty \\
\int_x |\Upsilon_{3,3}| \lesssim & \|q^{(k)}_{xyy} w \{1 - \chi\}\|_{x = x_0} |q^{\langle k \rangle} \{1 - \chi\} w \langle y \rangle^{-2}\|_{x = x_0} \| \bar{v}_{yyy}^{\langle k - 1 \rangle} \|_\infty \| \bar{u}_{y} \langle y \rangle^2 \|_\infty \\
\int_x |\Upsilon_{3,4}|\lesssim & \|q^{(k)}_{xyy} w \{1 - \chi\}\|_{x = x_0} \| \bar{v}_{yyy}^{\langle k - 1 \rangle} \|_\infty |q_y^{\langle k  \rangle} w \{1 - \chi\}\|_{x = x_0} \\
\int_x |\Upsilon_{3,5}| \lesssim & \|q^{(k)}_{xyy}\|_{x = x_0, loc} \|q_y^{\langle k  \rangle}\|_{x = x_0,loc} \| \bar{v}_{yyy}^{\langle k - 1 \rangle} \|_{\infty, loc}. 
\end{align*}

We now move to, using Lemma \ref{lemma:Sobolev} to control the $\bar{v}$ contribution:
\begin{align*}
\int_x |\Upsilon_{4}|\lesssim & \int_x ( q^{(k)}_{xyy}, \bar{u}^j_{yy} q_{yy}^{k-j} \bar{u}^2 w^2 \{1 - \chi\}^2) +  \int_x ( q^{(k)}_{xyy}, \bar{u}_{y}^j q_{yyy}^{k-j} \bar{u}^2 w^2 \{1 - \chi\}^2) \\
& + | \int_x ( q^{(k)}_{xyy}, \bar{u}_{y}^j q_{yy}^{k-j} \bar{u} \bar{u}_{y} w^2 \{1 - \chi\}^2) +  \int_x (q^{(k)}_{xyy}, \bar{u}_{y}^j q_{yy}^{k-j} \bar{u}^2 ww_{y} \{1 - \chi\}^2) \\
& +  \int_x ( q^{(k)}_{xyy}, \bar{u}_{y}^j q_{yy}^{k-j} \bar{u}^2 w^2 \{1 - \chi\} \chi') \\
\lesssim & \int_x \Big[ \|q^{(k)}_{xyy} w \{1 - \chi\}\| |q_{yy}^{\langle k \rangle} w \{1 - \chi\} \|_{x = x_0} \|\bar{v}^{\langle k - 1 \rangle}_{yy}  \|_\infty \\
& + \|q^{(k)}_{xyy} w \{1 - \chi\}\| |q_{yyy}^{\langle k \rangle} w \{1 - \chi\}\|_{x = x_0} \| \bar{v}^{\langle k- 1 \rangle}_{yy} \|_\infty \\
& + \|q^{(k)}_{xyy} w \{1 - \chi\}\|_{x = x_0} \|q_{yy}^{\langle k  \rangle} w \{1 - \chi\}\|_{x = x_0} \| \bar{v}_{yy}^{\langle k - 1 \rangle} \|_\infty \| \bar{u}_{y} \|_\infty \\
& + \|q^{(k)}_{xyy} w \{1 - \chi\}\|_{x = x_0} \|q_{yy}^{\langle k  \rangle} w  \{1 - \chi\}\|_{x = x_0} \| \bar{v}_{yy}^{\langle k - 1 \rangle}  \|_\infty  \\
& + \|q^{(k)}_{xyy}\|_{x = x_0,loc} \|q^{\langle k  \rangle} \|_{x = x_0,loc} \| \bar{v}_{yy}^{\langle k - 1 \rangle} \|_{\infty, loc} \Big] \\
\lesssim & o_L(1) p_k \| q \|_{X_{\langle k \rangle}}^2. 
\end{align*}

Next, since $\bar{u}_y$ is bounded by Theorem \ref{thm.Oleinik}, 
\begin{align*}
\int_x |\Upsilon_{5}| \lesssim & | \int_x ( v^{(k)}_{yyy}, q^{(k)}_{xyy} \{ \bar{u} \bar{u}_{y} w^2 \{1 - \chi\}^2 + \bar{u}^2 ww_{y} \{1 - \chi\}^2 + \bar{u}^2 w^2 \{1 - \chi\} \chi' \})| \\
\lesssim & \int_{x} \Big[ \|q^{(k)}_{xyy} w \{1 - \chi\}\|_{x = x_0} \|v^{(k)}_{yyy} w \{1 - \chi\} \|_{x = x_0} \| \bar{u}_{y}\|_\infty \\
& + \|q^{(k)}_{xyy} w \{1 - \chi\}\|_{x = x_0} \|v^{(k)}_{yyy} w \{1 - \chi\}\|_{x = x_0} \\
& + \|q^{(k)}_{xyy} w|_{x = x_0,loc} \|v^{(k)}_{yyy} \{1 - \chi\}|_{x = x_0,loc} \Big] \\
\lesssim & o_L(1) \| q \|_{X_{\langle k \rangle}}^2. 
\end{align*}

To conclude, we have 
\begin{align*}
\int_x |\Upsilon_{6}| = & |- \sum_{j =1}^k \int_x |\binom{k}{j} (q^{(k)}_{xyy}, \p_y \{ \bar{u}^2 \p_x^j \bar{u} \p_x^{k-j} q_{yyy} w^2 \{1 - \chi \}^2  \})_{x = x_0}| \\
\lesssim & \int_x \| q^{(k)}_{xyy}  w \{1 - \chi \}\|_{x = x_0} \| \bar{v}_{yy}^{\langle k - 1 \rangle} \|_\infty \| q^{\langle k - 1 \rangle}_{yyyy} \{1 - \chi \} w \|_{x = x_0} \\
\lesssim & o_L(1) p_k \| q \|_{X_{\langle k \rangle}}^2,
\end{align*}

\noindent all of which are acceptable contributions due to the cut-off $\{1 - \chi \}$. 

  This now concludes our treatment \eqref{belly:of:june} and consequently the left-hand side of \eqref{produce:produce}. We now move to the terms from the right-hand side of \eqref{produce:produce}, upon using Lemma \ref{lemma:Sobolev} for the $\bar{v}$ contribution, 
\begin{align*}
&\int_x \|  \p_{xy} \{ \p_x^{j_1}\bar{u} \p_x^{j_2} \bar{u} \p_x^{k-j} q_y \} w \{1 - \chi\}(x) \|_{x = x_0}^2 \\
&\lesssim \int_x \| \Big[ \p_x^{\langle k + 1 \rangle}\bar{u}_{y} \p_x^{\langle k \rangle}\bar{u} \p_x^{\langle k -1  \rangle} q_y + \p_x^{\langle k +1 \rangle} \bar{u} \p_x^{\langle k \rangle}\bar{u}_{y} \p_x^{\langle k - 1 \rangle} q_y \\
& \hspace{7 mm} + \p_x^{\langle k \rangle} \bar{u} \p_x^{\langle k \rangle} \bar{u} \p_x^{\langle k -1 \rangle} q_{yy}  + \p_x^{\langle k \rangle} \bar{u} \p_x^{\langle k \rangle} \bar{u}_{y} \p_x^{\langle k \rangle} q_y \\
& \hspace{7 mm}  + |\p_x^{\langle k \rangle} \bar{u}|^2 \p_x^{\langle k \rangle} q_{yy} \Big]  w \{1 - \chi\}(x) \|_{x = x_0}^2 \\
&\lesssim \int_x \Big[ \|\p_x^{\langle k - 1 \rangle} \bar{v}_{y} \|_\infty \|\p_x^{\langle k \rangle} \bar{v}_{yy} w \{1 - \chi\} \|_{x = x_0} \|\p_x^{\langle k - 1 \rangle} q_y w \{1 - \chi\} \|_{\infty}^2 \\
& \hspace{7 mm} + \| \p_x^{\langle k - 1 \rangle} \bar{v}_{yy}  \|_\infty^2 \| \p_x^{\langle k \rangle} \bar{v}_{y} \|_\infty^2 \|\p_x^{\langle k - 1 \rangle} q_y w \{1 - \chi\} \|_{x = x_0}^2 \\
& \hspace{7 mm} + \| \p_x^{\langle k - 1 \rangle} \bar{v}_{y} \|_\infty^2 \|\p_x^{\langle k - 1 \rangle} q_{yy} w \{1 - \chi\}\|_{x = x_0}^2  \\
& \hspace{7 mm} + \| \p_x^{\langle k-1 \rangle} \bar{v}_{y} \|_\infty \| \p_x^{\langle k - 1 \rangle} \bar{v}_{yy} \|_\infty^2 \|\p_x^{\langle k \rangle} q_y w \{1 - \chi\} \|_{x = x_0}^2 \\
& \hspace{7 mm} + \| \p_x^{\langle k - 1 \rangle} \bar{v}_{y} \|_\infty^2 |\p_x^{\langle k \rangle} q_{yy} w \{1 - \chi\} \|_{x = x_0}^2 \Big] \\
\lesssim & o_L(1) p_k \| q \|_{X_{\langle k \rangle}}^2.
\end{align*}

We now move to the $\Lambda$ terms: 
\begin{align*}
\p_x^k \Lambda = \sum_{j = 0}^k  \bar{v}^j_{xyy} I_x[v^{(k-j)}_y] + \bar{v}^j_{yy} v^{(k-j)}_{y} - \bar{v}^j_{x} I_x[v^{(k-j)}_{yyy}] - \bar{v}^j v^{(k-j)}_{yyy}
\end{align*}

We estimate directly:
\begin{align*}
&\| \bar{v}^j_{xyy} I_x[v^{(k-j)}_y]  w \{1 - \chi\}(x) \|_{x = x_0}^2 \lesssim \|\bar{v}^{\langle k \rangle}_{xyy} w  \{1 - \chi\}(x)\|_{x = x_0}^2 \| v_y^{\langle k  \rangle}   \|_\infty^2 \\
&\| \bar{v}^j_{yy} v^{(k-j)}_y  w \{1 - \chi\}(x) \|_{x = x_0}^2 \lesssim \|\bar{v}^{\langle k \rangle}_{xyy} w \{1 - \chi\}(x) \|_{x = x_0}^2 \| v_y^{\langle k  \rangle}   \|_\infty^2 \\
&\|  |\bar{v}^j_{x} I_x[v^{(k-j)}_{yyy} \{1 - \chi\} w (x) \|_{x = x_0}^2 \lesssim \| \bar{v}^{\langle k +1 \rangle} \|_\infty \| v^{\langle k \rangle}_{yyy} w \{1 - \chi\} (x) \|_{x = x_0}^2 \\
& \| \bar{v}^j |v^{(k-j)}_{yyy} \{1 - \chi\} w(x) \|_{x = x_0}^2  \lesssim \| \bar{v}^{\langle k +1 \rangle} \|_\infty \| v^{\langle k \rangle}_{yyy} w \{1 - \chi\} (x) \|_{x = x_0}^2.
\end{align*}

Upon integrating in $x$, the above terms are majorized by $o_L(1) p(\| \bar{q} \|_{X_{\langle k+ 1 \rangle}}) (1 + \| q \|_{X_k} )$, upon applying Lemma \ref{lemma:Sobolev} for $k+1$. We now move to the $U(u^0)$ terms: 
\begin{align*}
\int |\p_x^k U(u^0)|^2 w^2 \{1 - \chi\}^2 \le & \int \Big[ |\bar{v}^k_{xyy}|^2 |u^0|^2 + |\bar{v}^k_{x}|^2 |u^0_{yy}|^2 \Big] w^2 \{1 - \chi\}^2 \\
\le & \|u^0 \|_\infty^2 \|\bar{v}^k_{xyy} w \{1 - \chi\}\|_{x = x_0}^2 + \|u^0_{yy} w \{1 - \chi\}\|^2 \|\bar{v}^k_{x}\|_\infty^2. 
\end{align*}

Integrating, the above is majorized by $C(u^0) o_L(1)  p(\| \bar{q} \|_{X_{\langle k +1 \rangle}})$, upon applying Lemma \ref{lemma:Sobolev} for $k+1$. Similarly, the $g$ contributions are clearly estimated via $\|\p_{xy} g_1^{k} \{1 - \chi\} w \|^2$.
\end{proof}

\begin{proposition} \label{Prop.1}   For $k \ge 0$, and let $q$ solve (\ref{origPrLay.beta}). Then:
\begin{align} 
\begin{aligned} \label{iprob}
\| q \|_{X_{ k }} \lesssim & C(q_0) + \| \p_x^k \p_{xy}g_1 w \|^2 + o_L(1) \| \p_x^k \p_x \p_{xy}g_1 \langle y \rangle \|^2 + \kappa o_L(1) C(u^0) \\
& + o_L(1)(p(\| \bar{q} \|_{X_{\langle k \rangle}}) + \kappa p(\| \bar{q} \|_{X_{\langle k + 1 \rangle}}) ) \Big( C(q_0) + \| q \|_{X_{k}}^2 \Big).
\end{aligned}
\end{align}
\end{proposition}
\begin{proof} We add together (\ref{Ek}), a small multiple of (\ref{solid.py4.cof}) and (\ref{weight.H4.cof}). On the left-hand side, this produces 
\begin{align}
\begin{aligned}
&\sup [ |q^{(k)}_{yyy}w \{1-\chi\}|^2 + |\bar{u} q^{(k)}_{xy}|^2] + \| v^{(k)}_{yyyy} w \{1-\chi\} \|^2 \\
&+ \| v^{(k)}_{yyyy} \|_{loc} + \| q^{(k)}_{xyy} w \{1-\chi\} \|^2 + \| q^{(k)}_{xyy} \sqrt{\bar{u}} \|^2,
\end{aligned}
\end{align}

\noindent which can clearly be combined to majorize $\| q^{(k)} \|_X$. On the right-hand side 
\begin{align}
\begin{aligned}
&\| \p_{xy}\p_x^k g_1 w \{1-\chi\} \|^2 + C(q_0) + \kappa o_L(1) C(u^0) + o_L(1)(p(\| \bar{q} \|_{X_k}) \\
&+ \kappa p(\| \bar{q} \|_{X_{\langle k+ 1 \rangle}})(C(q_0) + \| q^{(k)} \|_X^2) + o(1) \| q^{(k)} \|_{\mathcal{E}} \\
&+ o_L(1) \| \p_{xxy}\p_x^k g_1 \langle y \rangle \|^2 + |\bar{u} q_{xy}^k(0,\cdot)|^2
\end{aligned}
\end{align}

\noindent Of these, the $o(1) \| q^{(k)} \|_{\mathcal{E}}$ term is absorbed to the left-hand side.  Finally, the initial value $|\bar{u} q^{(k)}_{xy}(0,\cdot)|^2$ is obtained through (\ref{whois}).
\end{proof}

We can upgrade to higher $y$ regularity by using the equation. In this direction, we establish the following lemma: 
\begin{lemma} \label{lemma:higher:y} Let $q$ solve (\ref{origPrLay.beta}). Then the following inequality is valid for any $k \ge 0$, 
\begin{align}
\| \p_y^{j} v \| \lesssim \| \p_y^{j - 4} F \| +  p_{\langle k \rangle} \| q \|_{X_{\langle k \rangle}} + C(u^0), \qquad 0 \le j \le 4 + 2k.  
\end{align}
\end{lemma}
\begin{proof}We will first address the case of $k = 1, j = 5$. We take $\p_y$ of equation \eqref{theta:approx:1} to obtain by using simply the definitions \eqref{def:Lambda} - \eqref{def:op:U}, 
\begin{align} \n
\| \p_y^5 v \| \le &\| F_y \| + \kappa o_L(1) \| \p_y U(u^0) \|_{L^2_y} + \kappa \| \p_y \Lambda(v) \| + \| \p_{xyy} (\bar{u}^2 q_y) \| \\
\lesssim & \| F_y \| + \kappa o_L(1) C(u^0) + \kappa p_{\langle 1 \rangle} \| q \|_{X_{\langle 1 \rangle}} + \| \p_{xyy}(\bar{u}^2 q_y) \|. 
\end{align}
We next expand using the product rule and estimate using Lemma \ref{lemma:Sobolev},  
\begin{align} \n
\| \p_{xyy}(\bar{u}^2 q_y) \| \lesssim & \| 2 \bar{u} \bar{u}_x q_{yyy} \| + \| \bar{u}^2 q_{xyyy} \| + \| \p_y^2(\bar{u}^2) q_{xy} \| \\ \n
& + \| \p_{xyy}(\bar{u}^2) q_y \| + \| (\bar{u} \bar{u}_y)_x q_{yy} \| + \| \bar{u} \bar{u}_y q_{xyy} \| \\
\lesssim & \| q \|_{X_{\langle 1 \rangle}}
\end{align}

It is clear that we can upgrade to higher $y$ regularity by iterating the above. 
\end{proof}

We now come to the proofs of two of our main results. 
\begin{proof}\textit{of Proposition \ref{thm.diff} } Proposition \ref{thm.diff} is a direct consequence of Proposition \ref{Prop.1} and Lemma \ref{lemma:higher:y}.
\end{proof}

\begin{proof} \textit{of Theorem \ref{thm.main} } We begin by reformulating the Prandtl equations, \eqref{Pr.leading} into the D-Prandtl system, analogous to \eqref{origPrLay.beta}, which produces ($U = \Lambda = f = 0$)
\begin{align} \label{nl:Prandtl:q:form}
- \p_{xy} \{ \bar{u}^2 q_y \} + \p_y^4 v = 0, \hspace{3 mm} q := \frac{v}{\bar{u}}.
\end{align}

From here, Proposition \ref{Prop.1} is applied with $g_1 = 0, \kappa = 0$, and $q = \bar{q}$ to give 
\begin{align*}
\| q \|_{X_k} = \| \bar{q} \|_{X_k} \lesssim C(q_0) \lesssim C(\bar{U}_0). 
\end{align*}

\noindent Above, we have used that the constant $C(q_0)$ depends on $\| \bar{u} q_{xy} \|_{x = 0}$. From \eqref{nl:Prandtl:q:form}, we obtain 
\begin{align} \label{q:initial}
- \bar{u} q_{xy} = - \frac{v_{yyy}}{\bar{u}} + 2 \bar{u}_x q_y, 
\end{align}

\noindent from which  
\begin{align}
\| \bar{u} q_{xy} \|_{x = 0} \le \| \frac{1}{\bar{u}} v_{yyy} \|_{x = 0} + \| 2 \bar{u}_x q_{y} \|_{x = 0} \le C(\bar{U}_0).
\end{align}

\noindent It is important to note that $v_{yyy}|_{x = 0}(0) = 0$ from $\bar{u}_{yy}|_{y = 0} = 0$ according to the Prandtl equation, \eqref{Pr.leading}. Above, we use that the quantities $\p_y^N v|_{x= 0}$ for any $N \ge 0$ is determined according to the initial data, $u^0_p|_{x = 0}$. Hence, \eqref{nl:Prandtl:q:form} becomes 
\begin{align*}
\| q \|_{X_0} = \| \bar{q} \|_{X_0} \le C(u^0). 
\end{align*}
This concludes the proof. 
\end{proof}

\section{Construction of Approximate Navier-Stokes Solution}

\subsection{Specification of Equations}

We will assume the expansions: 
\begin{align}
&U^\eps = \tilde{u}^n_s + \eps^{N_0} u, \hspace{3 mm} V^\eps = \tilde{v}^n_s + \eps^{N_0} v, \hspace{3 mm} P^\eps = \tilde{P}^n_s + \eps^{N_0} P.
\end{align}

\noindent We will denote the partial expansions: 
\begin{align}
&u_s^i = \sum_{j = 0}^i \sqrt{\eps}^j u^j_e + \sum_{j = 0}^{i-1} \sqrt{\eps}^j u^j_p, \hspace{5 mm} \tilde{u}_s^i = u_s^i + \sqrt{\eps}^i u^i_p, \\
&v_s^i = \sum_{j = 1}^i \sqrt{\eps}^{j-1} v^j_e + \sum_{j = 0}^{i-1} \sqrt{\eps}^j v^j_p, \hspace{5 mm} \tilde{v}_s^i = v_s^i + \sqrt{\eps}^i v^i_p, \\
&P^i_s = \sum_{j = 0}^i \sqrt{\eps}^j P^j_e, \hspace{5 mm}  \tilde{P}_s^i = P_s^i + \sqrt{\eps}^i \Big\{ P^i_p + \sqrt{\eps} P^{i,a}_p \Big\}.
\end{align}

\noindent  We will also define $u^{E,i}_s = \sum_{j = 0}^i \sqrt{\eps}^j u^j_e$ to be the ``Euler" components of the partial sum. Similar notation will be used for $u^{P,i}_s, v^{E,i}_s, v^{P,i}_s$. The following will also be convenient: 
\begin{align}
\begin{aligned} \label{profile.splitting}
&u_s^E := \sum_{i =0}^n \sqrt{\eps}^i u^i_e, \hspace{3 mm} v_s^E := \sum_{i = 1}^n \sqrt{\eps}^{i-1} v^i_e, \\
&u_s^P := \sum_{i = 0}^n \sqrt{\eps}^i u^i_p, \hspace{3 mm} v_s^P := \sum_{i = 0}^n \sqrt{\eps}^i v^i_p, \\
&u_s = u_s^P + u_s^E, \hspace{3 mm} v_s = v_s^P + v_s^E. 
\end{aligned}
\end{align} 

\noindent  The $P^{i,a}_p$ terms are ``auxiliary Pressures" in the same sense as those introduced in \cite{GN} and \cite{Iyer} and are for convenience. We will also introduce the notation: 
\begin{align} \label{bar.defs}
\bar{u}^i_p := u^i_p - u^i_p|_{y = 0}, \hspace{5 mm} \bar{v}^i_p := v^i_p - v^i_p(x,0), \hspace{5 mm} \bar{v}^i_e = v^i_e - v^i_e|_{Y = 0}.
\end{align}

We first record the properties of the leading order $(i = 0)$ layers. For the outer Euler flow, we will take a shear flow, $[u^0_e(Y), 0, 0]$. The derivatives of $u^0_e$ decay rapidly in $Y$ and that is bounded below, $|u^0_e| \gtrsim 1$. 

For the leading order Prandtl boundary layer, the equations are given in (\ref{Pr.leading}), for the $i$'th Euler layer, $i \ge 1$, the equations are given by (\ref{des.eul.1.intro}), whereas for the $i$'th Prandtl layer the equations are given by (\ref{des.pr.1.intro}).

The relevant definitions of the forcing terms in those equations are given below. Note that as a matter of convention, summations that end with a negative number are empty sums.
\begin{definition}[Forcing Terms] \label{def.forcing}  
\begin{align*}
&-f^i_{E,1} := u^{i-1}_{ex} \sum_{j = 1}^{i-2} \sqrt{\eps}^{j-1} \{u^j_e + u^j_p(x,\infty) + u^{i-1}_e \sum_{j = 1}^{i-2} \sqrt{\eps}^{j-1} u^j_{ex} \\
& \hspace{15 mm} + \sqrt{\eps}^{i-2}[ \{u^{i-1}_{e} + u^{i-1}_p(x,\infty) \} u^{i-1}_{ex} + v^{i-1}_e u^{i-1}_{eY}] \\
& \hspace{15 mm} + u^{i-1}_{eY} \sum_{j = 1}^{i-2} \sqrt{\eps}^{j-1} v^j_e + v^{i-1}_e \sum_{j = 1}^{i-2} \sqrt{\eps}^{j-1} u^j_{eY} - \sqrt{\eps} \Delta u^{i-1}_e, \\
&-f^i_{E,2} := v^{i-1}_{eY} \sum_{j = 1}^{i-2} \sqrt{\eps}^{j-1} v^j_e + v_e^{i-1} \sum_{j = 1}^{i-2} \sqrt{\eps}^{j-1} v^j_{eY} + \sqrt{\eps}^{i-2}[v^{i-1}_e v^{i-1}_{eY} + u^{i-1}_e v^{i-1}_{ex}] \\
& \hspace{15 mm} + \{u_e^{i-1} + u^{i-1}_p(x,\infty)\} \sum_{j =1}^{i-2} \sqrt{\eps}^{j-1} v^j_{ex} + v^{i-1}_{ex} \sum_{j = 1}^{i-2} \sqrt{\eps}^{j-1} \{u^j_e + u^j_p(x,\infty) \}\\
& \hspace{15 mm} - \sqrt{\eps} \Delta v^{i-1}_e, \\
&-f^{(i)} := \sqrt{\eps} u^{i-1}_{pxx} + \eps^{-\frac{1}{2}} \{ v^i_e - v^i_e(x,0) \} u^0_{py} + \eps^{-\frac{1}{2}} \{ u^0_e - u^0_e(0) \} u^{i-1}_{px} + \eps^{-\frac{1}{2}} \{ u^{P, i-1}_{sx} \\
& \hspace{15 mm} - \bar{u}^0_{sx} \} u^{i-1}_p +  \eps^{-\frac{1}{2}} \{ u^{E, i-1}_{sx}  - \bar{u}^0_{sx} \} \{u^{i-1}_p - u^{i-1}_p(x,\infty) \} + \eps^{-\frac{1}{2}} v^{i-1}_p \{ \bar{u}^{i-1}_{sy} \\
& \hspace{15 mm} - u^0_{py} \} + u^{i-1}_{px} \sum_{j =1 }^{i-1} \sqrt{\eps}^{j-1}(u^j_e + u^j_p)  + \eps^{-\frac{1}{2}} (v_s^{i-1} - v_s^1) u^{i-1}_{py} + \eps^{-\frac{1}{2}} (v^1_e \\
& \hspace{15 mm} - v^1_e(x,0)) u^{i-1}_{py} + \sqrt{\eps} u^i_{eY} \sum_{j = 0}^{i-1} \sqrt{\eps}^j v^j_p + v^i_e \sum_{j = 1}^{i-1} \sqrt{\eps}^{j-1} u^j_{py} + u^i_{ex} \sum_{j = 0}^{i-1} \sqrt{\eps}^j \{u^j_p \\
& \hspace{15 mm} - u^j_p(x,\infty) \} + u^i_e \sum_{j = 0}^{i-1} \sqrt{\eps}^j u^j_{px} + \int_y^\infty \p_x \{ \sqrt{\eps}^2 u^i_e \sum_{j = 0}^{i-1} \sqrt{\eps}^j v^j_{px} + \sqrt{\eps} v^i_{ex} \\
& \hspace{15 mm} \times \sum_{j = 0}^{i-1} \sqrt{\eps}^j \{u^j_p - u^j_p(x,\infty)\} + \sqrt{\eps}^2 v^i_{eY} \sum_{j = 0}^{i-1} \sqrt{\eps}^j v^j_p + \sqrt{\eps} v^i_e \sum_{j = 0}^{i-1} \sqrt{\eps}^j v^j_{py} \\
& \hspace{15 mm} + \sqrt{\eps} v^{i-1}_s v^{i-1}_{py} + \sqrt{\eps} v_{sy}^{i-1} v^{i-1}_p  + \sqrt{\eps} v^{E,i-1}_{sx} \{u^{i-1}_p - u^{i-1}_p(x,\infty)\} \\
& \hspace{15 mm} +  \sqrt{\eps} v^{P,i-1}_{sx} u^{i-1}_{p} + \sqrt{\eps} u_s^{i-1} v^{i-1}_{px} + \sqrt{\eps} \Delta_\eps v^{i-1}_p + \sqrt{\eps}^i \{ u^{i-1}_p v^{i-1}_{px} + v^{i-1}_p v^{i-1}_{py} \} \} \ud z.
\end{align*}
\end{definition}

For $i = 1$ only, we make the following modifications. The aim is to retain only the required order $\sqrt{\eps}$ terms into $f^{(1)}$. $f^{(2)}$ will then be adjusted by including the superfluous terms. Define: 
\begin{align} 
\begin{aligned} \label{defn.f1.special}
f^{(1)} := & - u^0_p u^1_{ex}|_{Y = 0} - u^0_{px} u^1_e|_{Y = 0}  - \bar{u}^0_{eY}(0) y u^0_{px} - v^0_p u^0_{eY} - v^1_{eY}(0) y u^0_{py}.
\end{aligned}
\end{align} 

For the final Prandtl layer, we must enforce the boundary condition $v^n_p|_{y = 0} = 0$. Define the quantities $[u_p, v_p, P_p]$ to solve
\begin{align} \label{des.pr.1}
\left.
\begin{aligned}
&\bar{u} \p_x u_p + u_p \p_x \bar{u} + \p_y \bar{u} v_p + \bar{v} \p_y u_p + \p_x P_p - \p_{yy} u_p := f^{(n)}, \\  
& \p_x u_p + \p_y v_p = 0,  \hspace{5 mm} \p_y P^i_p = 0\\  
& [u_p, v_p]|_{y = 0} = [-u^n_e, 0]|_{y = 0}, \hspace{5 mm} u_p|_{y \rightarrow \infty} = 0 \hspace{5 mm} v_p|_{x = 0} = V_P^n. 
\end{aligned}
\right\}
\end{align}

\noindent Note the change in boundary condition of $v_{p}|_{y = 0} = 0$ which contrasts the $i = 1,..,n-1$ case. This implies that $v_p = \int_0^y u_{px} \ud y'$. For this reason, we must cut-off the Prandtl layers: 
\begin{align*}
&u^n_p := \chi(\sqrt{\eps}y) u_p + \sqrt{\eps} \chi'(\sqrt{\eps}y) \int_0^y u_p(x, y') \ud y', \\
&v^n_p := \chi(\sqrt{\eps}y) v_p. 
\end{align*}

Here $\mathcal{E}^n$ is the error contributed by the cut-off: 
\begin{align*}
\mathcal{E}^{(n)} &:= \bar{u} \p_x u^{n}_{p} + u^n_p \p_x \bar{u}  +\bar{v} \p_y u^n_{p} + v^n_p \p_y \bar{u}  - u^n_{pyy} - f^{(n)}. 
\end{align*}

Computing explicitly: 
\begin{align} \n
\mathcal{E}^{(n)} := &(1-\chi) f^{(n)} + \bar{u} \sqrt{\eps} \chi'(\sqrt{\eps}y) v_p(x,y) + \bar{u}_{x} \sqrt{\eps} \chi' \int_0^y u_p \\  \n
& + \bar{v} \sqrt{\eps} \chi' u_p + \eps \bar{v} \chi'' \int_0^y u_p + \sqrt{\eps} \chi' u_p \\ \label{dan.1}
& + \eps^{\frac{3}{2}} \chi''' \int_0^y u_p + 2\eps \chi'' u_p + \sqrt{\eps} \chi' u_{py}.
\end{align}

We will now define the contributions into the next order, which will serve as the forcing for the remainder term: 
\begin{align} \n
&\underbar{f}^{(n+1)} := \sqrt{\eps}^n \Big[ \eps u^n_{pxx} + v^n_p\{ \bar{u}^n_{sy} - u^0_{py} \} + \{u^0_e - u^0_e(0) \} u^n_{px} \\ \n
& \hspace{15 mm} + u^n_{px} \sum_{j = 1}^n  \sqrt{\eps}^j (u^j_e + u^j_p) + \{ u^n_{sx} - \bar{u}^0_{sx} \} u^n_p + (v^n_s - v^1_s) u^n_{py} \\ \n
& \hspace{15 mm} + \{ v^1_e - v^1_e(x,0) \} u^n_{py} \Big] + \sqrt{\eps}^n \mathcal{E}^{(n)} + \sqrt{\eps}^{n+2} \Delta u^n_e \\ \label{underbar.f}
& \hspace{15 mm} + \sqrt{\eps}^n u^n_{ex} \sum_{j = 1}^{n-1} \sqrt{\eps}^j u^j_e + \sqrt{\eps}^n u^n_e \sum_{j = 1}^{n-1} \sqrt{\eps}^j u^j_{ex} + \sqrt{\eps}^{2n} [ u^n_e u^n_{ex} \\ \n
& \hspace{15 mm} + v^n_e u^n_{eY}] + \sqrt{\eps}^{n+1} u^n_{eY} \sum_{j= 1}^{n-1} \sqrt{\eps}^{j-1} v^j_e + \sqrt{\eps}^{n-1}v^n_e \sum_{j = 1}^{n-1} \sqrt{\eps}^{j+1} u^j_{eY},
\end{align}
and
\begin{align} \n
& \underbar{g}^{(n+1)} := \sqrt{\eps}^n \Big[ v_s^n \p_y v^n_p + \p_y v_s^n v^n_p + \p_x v^n_s u^n_p + u^n_s \p_x v^n_p - \Delta_\eps v^n_p \\ \n
& \hspace{15 mm}  + \sqrt{\eps}^n \Big( u^n_p \p_x v^n_p + v^n_p \p_y v^n_p \Big) \Big] + (\sqrt{\eps})^n \p_Y v^n_e \sum_{j = 1}^{n-1} (\sqrt{\eps})^{j-1} v^j_e \\ \label{underbar.g}
& \hspace{15 mm}+ \sqrt{\eps}^{n-1} v^n_e \sum_{j = 1}^{i-1} \sqrt{\eps}^j \p_Y v^j_e + \sqrt{\eps}^{2n-1} [v^n_e v^n_{eY} + u^n_e \p_x v^n_e]  \\ \n
& \hspace{15 mm} + \sqrt{\eps}^n u^n_e \sum_{j =1}^{n-1} (\sqrt{\eps})^{j-1}\p_x v^j_e + \sqrt{\eps}^{n-1} \p_x v^n_e \sum_{j = 0}^{n-1} \sqrt{\eps}^j u^j_e + \sqrt{\eps}^{n+1} \Delta v^n_e.
\end{align}
Finally, we have
\begin{align} \label{forcingdefn}
&F_R := \eps^{-N_0} ( \p_y \underbar{f}^{(n+1)} - \eps \p_x \underbar{g}^{(n+1)}).
\end{align}

\subsection{Construction of Euler Layers} \label{appendix.Euler}

Our starting point is the system (\ref{des.eul.1.intro}). Going to vorticity yields the system we will analyze:  
\begin{align} 
\begin{aligned} \label{eqn.vort}
&u^0_e \Delta v^i_e + u^0_{eYY} v^i_e = F^{(i)} := \p_Y f^i_{E,1} - \p_x f^i_{E,2}, \\
&v^i_e|_{Y = 0} = - v^{i-1}_p|_{y = 0}, \hspace{3 mm} v^i_e|_{x = 0,L} = V^i_{E, \{0, L\}}, \hspace{3 mm} u^i_e|_{x = 0} = U^i_{E,0}.
\end{aligned}
\end{align}

\noindent The data for $u^i_e|_{x = 0}$ is required because $u^i_e = u^i_e|_{x = 0} - \int_0^x v^i_{eY}$ will be recovered through the divergence free condition upon constructing $v^i_e$. 

We will quantify the decay rates as $Y \uparrow \infty$ for the quantities $V^i_{E,{0, L}}$ and $F^{(i)}$. 
\begin{definition} \label{def.wm1}   In the case of $i = 1$, define $w_{m_1} = Y^{m_1}$ if $v^1_{e}|_{x = 0} \sim Y^{-m_1}$ or $w_{m_1} = e^{m_1 Y}$ if $v^1_e|_{x = 0} \sim e^{-m_1 Y}$ as $Y \uparrow \infty$. This now fixes whether or not $w_{m}$ will refer to polynomial or exponential growth rates. For other layers, we will assume: 
\begin{align}
\begin{aligned}  \label{decay.rates.euler}
&V^i_{E,\{0, L\}} \sim w_{m_i}^{-1} \text{ for } m_i >> m_1 \\
&F^{(i)} \sim w_{l_i}^{-1} \text{ for some } l_i >> 0. 
\end{aligned}
\end{align}
Finally, let $m_i' := \min \{ m_i, l_i \}$. 
\end{definition}

Define: 
\begin{align} \label{def:S}
S(x,Y) = (1 - \frac{x}{L}) \frac{V_{i,0}(Y)}{v^{i-1}_p(0,0)} v^{i-1}_p(x,0) + \frac{x}{L}\frac{V_{i,L}(Y)}{v^{i-1}_p(L,0)}v^{i-1}_p(x,0),
\end{align}

\noindent and consider the new unknown:
\begin{align*}
\bar{v} := v^i_e - S, 
\end{align*}

\noindent which satisfies the Dirichlet problem: 
\begin{align}
-u^0_e \Delta \bar{v} + u^0_{eYY} \bar{v} = F^{(i)} + \Delta S, \hspace{5 mm} \bar{v}|_{\p \Omega} = 0.
\end{align}

From here, we have for any $m < m_i' - n_0$ for some fixed $n_0$, perhaps large, 
\begin{align}
||v \cdot w_m||_{H^1} \lesssim 1.
\end{align}

To go to higher-order estimates, we must invoke that the data are well-prepared in the following sense: taking two $\p_Y^2$ to the system yields:
\begin{align} \label{corn.1}
&\p_Y^2 v^i_e(0,Y) = \p_Y^2 V_{i,0}(Y), \\ \label{corn.2}
&\p_Y^2 v^i_e(L,Y) = \p_Y^2 V_{i,L}(Y), \\ \label{corn.3}
&\p_Y^2 v^i_e(x,0) = \frac{1}{u^0_e(0)} \Big\{ v^{i-1}_{pxx}(x,0) + u^0_{eYY}(0) v^{i-1}_p(x,0) + F^{(i)}(x,0) \Big\}.
\end{align}

Our assumption on the data, which are compatibility conditions, ensure: 
\begin{align}
&\p_Y^2 V_{i,0}(0) = \frac{1}{u^0_e(0)} \Big\{ v^{i-1}_{pxx}(0,0) + u^0_{eYY}(0) v^{i-1}_p(0,0) + F^{(i)}(0,0) \Big\}, \\
&\p_Y^2 V_{i,0}(L) = \frac{1}{u^0_e(0)} \Big\{ v^{i-1}_{pxx}(L,0) + u^0_{eYY}(0) v^{i-1}_p(L,0) + F^{(i)}(L,0) \Big\}.
\end{align}

It is natural at this point to introduce the following definition:
\begin{definition}[Well-Prepared Boundary Data]    \label{def.well.prp} Consider the corner $(0,0)$. There exists a value of $\Big(\p_Y^2 v^i_e|_{Y = 0}\Big)|_{x = 0}$ which is obtained by evaluating (\ref{corn.3}) at $x = 0$. There exists a value of $\Big(\p_Y^2 v^i_e|_{x = 0} \Big)|_{Y = 0}$ which is obtained by evaluating (\ref{corn.1}) at $Y = 0$. These two values should coincide. The analogous statement should also hold for the corner $(L,0)$. In this case, we say that the boundary data are ``well-prepared to order 2". The data are ``well-prepared to order $2k$" if we can repeat the procedure for $\p_Y^{2k}$.
\end{definition}

We thus have the following system:
\begin{align} \n
-u^0_e \Delta v^1_{eYY} + u^0_{eYY} v^1_{eYY} &+ \p_Y^4 u^0_e v^1_e + 2 \p_Y^3 u^0_e v^1_{eY} \\ \label{eul.sys.1} &- 2 u^0_{eY} \Delta v^1_{eY} - u^0_{eYY} \Delta v^i_e = \p_{YY} F^{(i)}.
\end{align}

We can define another homogenization in the same way:
\begin{align}
S_{(2)}(x,Y) = (1 - \frac{x}{L}) \frac{V''_{i,0}(Y)}{\p_Y^2 v^i_{e}(x,0)} \p_Y^2 v^{i-1}_p(x,0) + \frac{x}{L}\frac{V''_{i,L}(Y)}{\p_Y^2 v^{i-1}_p(L,0)}\p_Y^2 v^{i-1}_p(x,0),
\end{align}

\noindent which is smooth and rapidly decaying by the assumption that the data are well-prepared. Let us consider the system for $\bar{v} := v^1_{eYY} - B_{(2)}$. The first step is to rewrite:
\begin{align} \n
v^i_{exx} &= -v^i_{eYY} + \frac{u^0_{eYY}}{u^0_e} v^i_e + F^i = - \bar{v} + S_{(2)} + F^i.
\end{align}

We can now rewrite the system (\ref{eul.sys.1}) in terms of $\bar{v}$:
\begin{align} \n
-u^0_e \Delta \bar{v} &+ u^0_{eYY} \bar{v} + \p_Y^4 u^0_e v^1_e + 2\p_Y^3 u^0_e v^1_{eY} \\ \n
& -2u^0_{eY} (\p_Y\{ \bar{v} + S_2 \} + \p_Y \{ \bar{v} + S_2 + F^i \}) \\ 
& - u^0_{eYY}  F^i  = u^0_e \Delta S_2 + u^0_{eYY} S_2 + \p_{YY}F^i.
\end{align}

Obtaining estimates for $\bar{v}$ yields for any $m < m_i' - n_0$: 
\begin{align}
||\bar{v} \cdot w_m||_{H^1} \lesssim 1.
\end{align}

Translating to the original unknown gives:
\begin{align} \label{H3eul.est.1}
||v^i_{eYY}, v^1_{eYYx}, v^i_{eYYY} \cdot w_M||_{L^2} \lesssim 1.
\end{align}

Using the equation and Hardy in $Y$, we can obtain: 
\begin{align} \label{H3eul.est}
||v^i_{exx}, v^i_{exxx}, v^1_{exY}, v^i_{exxY} \cdot w_M||_{L^2} \lesssim 1.
\end{align}

Thus, we have the full $H^3$ estimate. $u^1_e$ can be recovered through the divergence free condition:
\begin{align} \label{u1.rec}
u^i_e(x,Y) := u^i_e(0,Y) - \int_0^x \p_Y v^i_e(x', Y) \ud x'.
\end{align}

The compatibility conditions can be assumed to arbitrary order by iterating this process, and thus we can obtain:
\begin{proposition}   \label{L.e.constr} There exists a unique solution $v^i_e$ satisfying (\ref{eqn.vort}). With $u^i_e$ defined through (\ref{u1.rec}), the tuple $[u^i_e, v^i_e]$ satisfy the system (\ref{des.eul.1.intro}). For any $k \ge 0$ and $M \le m_i' - n_0$ for some fixed value $n_0 > 0$: 
\begin{align} \label{Euler:Hk:estimate}
||\{u^i_e, v^i_e \} w_M ||_{H^k} \le C_{k,M}.
\end{align}
\end{proposition}
\begin{proof}
The existence follows from Lax-Milgram, whereas the estimates follow from continuing the procedure resulting in (\ref{H3eul.est.1}) - (\ref{H3eul.est}).
\end{proof}

\begin{corollary} \label{Cor.cheap.quotient}   Assume $m_i >> m_1$ for $i = 2,...,n$. Then: 
\begin{align} \label{cheap.quotient}
\| \{ u^i_e, v^i_e \} w_{\frac{m_1}{2}} \|_{H^k} \lesssim 1. 
\end{align}
\end{corollary}
\begin{proof} This follows from two points. First, for the $i = 1$ case, the forcing is absent and therefore the parameter $l_1$ can be taken arbitrarily large. In particular this implies that $m_1' = m_1$. Second, a subsequent application of the above proposition shows that the $i$-th layer quantities decay like $m_1 - n_0$. An examination of the forcing terms $f^i_{E,1}, f^i_{E,2}$ shows that these quantities decay as $w_{m_1 - n_0}^{-1}$. Thus, for $i \ge 2$, we can take the parameter $l_i = m_1 - n_0 = m_i'$. Therefore, if $m_1$ is sufficiently large, $\frac{m_1}{2} << m_1 - 10 n_0$.
\end{proof}

Recall the definition of $m_i$ from Definition \ref{def.wm1}. The main estimate here is: 
\begin{lemma} \label{Quotient.Lemma}   Let $v^i_e$ be a solution to (\ref{eqn.vort}). For any $m' < m_i' := \min\{m_i, l_i\}$. Then, the following estimate holds: 
\begin{align}
\| v^i_e w_{m'} \|_{W^{k,\infty}} \lesssim 1.
\end{align}
\end{lemma}

\begin{proof} We first homogenize $v^i_e$ by introducing $\bar{v}_e := v^i_e - S$, where $S$ was defined in \eqref{def:S}.  Recall the definition of $\chi$ in (\ref{basic.cutoff}). We will localize using $1 -\chi(\frac{Y}{N})$ for some large, fixed $N > 1$. A direct computation produces the following: 
\begin{align*}
\Delta ( \{ 1 - \chi(\frac{Y}{N}) \} \bar{v}_e) =& -\{ 1 - \chi(\frac{Y}{N}) \} \frac{u^0_{eYY}}{u^0_e} v^1_e - \{ 1 - \chi(\frac{Y}{N}) \} \Delta S \\
& + 2 \p_Y \{ 1 - \chi(\frac{Y}{N}) \} \bar{v}_{eY} + \p_{YY} \{ 1 - \chi(\frac{Y}{N}) \} \bar{v}_e \\
& + (1 - \chi(\frac Y N)) \frac{F^{(i)}}{u^0_e} := R.
\end{align*}

Let $w = w_{m'}^{-1}$. Now we define the quotient $q^\delta = \frac{\{ 1 - \chi(\frac{Y}{N}) \} \bar{v}_e}{w(Y) + \delta}$, which satisfies: 
\begin{align} \label{difrn}
&\underbrace{\Delta q^\delta + 2 \frac{w_Y}{w + \delta} q^\delta_Y}_{:= \mathcal{T}_\delta} + \frac{w_{YY}}{w+\delta} q^\delta = \frac{R}{w+\delta}.
\end{align}

\noindent  \textit{Case 1: $w_{m_i}$ are polynomials in $Y$.} The following inequalities hold, independent of small $\delta$: 
\begin{align} \label{meeshka.1}
&|\frac{R}{w+\delta}| \lesssim 1, \qquad |\frac{w_{YY}}{w+\delta}\{ 1 - \chi(\frac{Y}{N}) \}| \le o(1). 
\end{align}

\noindent The second inequality above holds because $|w_{YY}| \lesssim Y^{-2} |w|$ for polynomial decay, so by taking $N$ large, we can majorize the desired quantity by $o(1)$.  To apply the maximum principle to $q^\delta$, we introduce the following barrier, for $m$ large and fixed and for $f = f(x) := 10  - (1+x)^2$ (which we note satisfies $f''(x) < - 1$):
\begin{align*}
q_{\pm}^\delta := q^\delta \pm f(mx) ( \sup |\frac{R}{w+\delta}| + \sup|\frac{w_{YY}}{w} q^\delta| ).
\end{align*}
Immediate computations gives $\mathcal{T}_\delta[q_-^\delta] \ge 0$ and $\mathcal{T}_\delta[q_+^\delta] \le 0$. Due to the homogenization, $\bar{v}_e|_{x = 0} = \bar{v}_e|_{x = L} = 0$, which immediately implies that $q^\delta_-|_{x = 0} < 0$ and $q^\delta_-|_{x = L} < 0$, and $q^\delta_+|_{x = 0} > 0$, $q^\delta_+|_{x = L} > 0$. Due to the presence of the cut-off, $q^\delta_{\pm}|_{Y = \frac{N}{2}} = 0$. Thus, applying the maximum principle to both $q_-^\delta, q_+^\delta$ on the domain $(x, Y) \in (0, L) \times (\frac N 2, \infty)$, gives: 
\begin{align*}
\| q^\delta \|_\infty \lesssim \sup|\frac{R}{w+\delta}| + \sup |\frac{w_{YY}}{w}q^\delta|. 
\end{align*}
Applying both estimates in (\ref{meeshka.1}) gives $\| q^\delta \|_\infty \lesssim 1$ uniformly in $\delta > 0$. Due to the cutoff $\{1 - \chi(\frac{Y}{N})\}$, all quantities are supported away from $Y = 0$, we may differentiate the equation, (\ref{difrn}), in $Y$ to obtain the new system: 
\begin{align*}
\Delta q^\delta_Y + 2 \frac{w_Y}{w+\delta} q^\delta_{YY} + [\frac{w_{YY}}{w+\delta} + 2 \p_Y\{ \frac{w_Y}{w+\delta} \}] q^\delta_Y = \p_Y\{ \frac{R}{w+\delta} \} - \p_Y\{ \frac{w_{YY}}{w+\delta} \} q^\delta.
\end{align*}

Clearly, we may repeat the above argument for the unknown $q^\delta_Y$. Bootstrapping further to $q^{\delta}_{YY}$ and using the equation, we establish: 
\begin{align*}
\|q^\delta_{xx} \|_\infty \lesssim 1. 
\end{align*}

For each fixed $Y$, $q^\delta_x(x_\ast, y) = 0$ for some $x_\ast = x_\ast(Y) \in [0,L]$ since $q^\delta(0,Y) = q^\delta(L,Y) = 0$. Thus, using the Fundamental Theorem of Calculus $\| q_x^\delta \|_\infty \lesssim 1$. Finally, we use the pointwise in $Y$ equality: 
\begin{align*}
|q^\delta_x - q_x| = \delta | \frac{\chi \bar{v}_{ex}}{w(w+\delta)} | \rightarrow 0 \text{ as } \delta \downarrow 0, \text{ pw in } Y.
\end{align*} 

Thus, for each fixed $Y$, there exists a $\delta_\ast = \delta_\ast(Y) > 0$ such that for $0 < \delta < \delta_\ast(Y)$, $|q^\delta_x - q_x| \le 1/2$. Thus, $|q_x(Y)| \lesssim 1$. This is true for all $Y$. Thus, $\| q_x \|_\infty \lesssim 1$.

\noindent \textit{Case 2: $w_{m_i}$ are exponential in $Y$.} In this case, we start with (\ref{difrn}), and perform $H^k$ energy estimates. We replace (\ref{meeshka.1}) with: 
\begin{align}
\| \frac{R}{w+\delta} \langle Y \rangle^M \| < \infty \text{ for large } M. 
\end{align}

From here, straightforward energy estimates show $\| q^\delta \|_{H^k} \lesssim 1$ for any $k$ as in estimate \eqref{Euler:Hk:estimate}. This is achieved by repeatedly differentiating in $Y$ and using that the cutoff $\{1 - \chi(\frac{Y}{N}) \}$ localizes away from the boundary $\{Y = 0\}$. We thus conclude $\| q^\delta_{xx} \|_{\infty} \lesssim \| q^\delta \|_{H^4} \lesssim 1$ using Sobolev embedding. The proof then concludes as in the polynomial case. 
\end{proof}

\begin{proof}[Proof of Theorem \ref{thm.construct}] Theorem \ref{thm.construct} is a consequence of Theorem \ref{thm.diff}, Proposition \ref{L.e.constr}, Corollary \ref{Cor.cheap.quotient}, and Lemmas \ref{lemma:higher:y}, \ref{Quotient.Lemma}.
\end{proof}

\noindent \textbf{Acknowledgements:} The research of Yan Guo is supported in part by NSF grants DMS-1611695, DMS-1810868. Sameer Iyer was also supported in part by NSF grant DMS 1802940.

\def\bibindent{3.5em}

\end{document}